\documentclass[11pt]{amsart}
\usepackage{amsmath,amssymb}
\usepackage[dvips]{graphicx}

\def\COMMENT#1{}
\def\TASK#1{}

\def\noproof{{\unskip\nobreak\hfill\penalty50\hskip2em\hbox{}\nobreak\hfill%
        $\square$\parfillskip=0pt\finalhyphendemerits=0\par}\goodbreak}
\def\endproof{\noproof\bigskip}
\newdimen\margin   % needed for macros \textdisplay & \ltextdisplay
\def\textno#1&#2\par{%
    \margin=\hsize
    \advance\margin by -4\parindent
           \setbox1=\hbox{\sl#1}%
    \ifdim\wd1 < \margin
       $$\box1\eqno#2$$%
    \else
       \bigbreak
       \hbox to \hsize{\indent$\vcenter{\advance\hsize by -3\parindent
       \sl\noindent#1}\hfil#2$}%
       \bigbreak
    \fi}
\def\proof{\removelastskip\penalty55\medskip\noindent{\bf Proof. }}

\def\eps{\varepsilon}

\def\a{\alpha}
\def\b{\beta}
\def\d{\delta}
\def\g{\gamma}
\def\u{\underline}

\def\dB{\mathcal B _{n,k}}
\def\doB{\overline{\mathcal B} _{n,k}}

\newtheorem{firstthm}{Proposition}[section]
\newtheorem{thm}[firstthm]{Theorem}
\newtheorem{prop}[firstthm]{Proposition}
\newtheorem{lemma}[firstthm]{Lemma}
\newtheorem{cor}[firstthm]{Corollary}

\newtheorem{conj}[firstthm]{Conjecture}
\newtheorem{claim}[firstthm]{Claim}
\newtheorem{fact}[firstthm]{Fact}

\begin{document}
\title{Exact minimum degree thresholds for perfect matchings in uniform hypergraphs}
\author{Andrew Treglown and Yi Zhao}
\date{\today}

\begin{abstract}
Given positive integers $k$ and $\ell$ where $4$ divides $k$ and $k/2 \leq \ell \leq k-1$,
we give a  minimum $\ell$-degree condition that ensures a perfect matching in a $k$-uniform hypergraph.
This condition is best possible and improves on work of
Pikhurko who gave an asymptotically exact result. Our approach makes use of the absorbing method,
as well as the hypergraph removal lemma and a structural result of Keevash and Sudakov relating to the
Tur\'an number of the expanded triangle.
\end{abstract}

\maketitle
\section{Introduction}\label{sec1}
A \emph{perfect matching} in a hypergraph $H$ is a collection of vertex-disjoint edges of $H$ which cover the vertex set $V(H)$ of $H$.
It is unlikely that there exists a characterisation of all those $k$-uniform hypergraphs that contain a perfect matching for $k\geq 3$.
Indeed, Garey and Johnson~\cite{garey} showed that the decision problem whether
a $k$-uniform hypergraph contains a perfect matching is NP-complete for $k\geq 3$.
(In contrast, a theorem of Tutte~\cite{tutte} gives a characterisation of all those graphs which contain a perfect matching.)
It is natural therefore to seek simple sufficient conditions that ensure a perfect matching in a $k$-uniform hypergraph.

Given a $k$-uniform hypergraph $H$ with an $\ell$-element vertex set $S$ (where $0 \leq \ell \leq k-1$) we define
$d_H (S)$ to be the number of edges containing $S$. The \emph{minimum $\ell$-degree $\delta _{\ell}
(H)$} of $H$ is the minimum of $d_H (S)$ over all $\ell$-element sets of vertices in $H$. Clearly $\delta_0(H)$ is the number of edges in $H$. We also
refer to  $\delta _1 (H)$ as the \emph{minimum vertex degree} of $H$ and  $\delta _{k-1}
(H)$ the \emph{minimum codegree} of $H$.

One of the earliest results on perfect matchings was given by Daykin and H\"aggkvist \cite{dayhag}, who showed that
a $k$-uniform hypergraph $H$ on $n$ vertices contains a perfect matching provided that $\delta_1(H)\ge (1 - 1/k)\binom{n-1}{k-1}$.
Recently there has been much interest in establishing minimum $\ell$-degree thresholds  that force a perfect matching in
a $k$-uniform hypergraph. See~\cite{rrsurvey} for a survey on matchings (and Hamilton cycles) in hypergraphs. In particular,
R\"odl, Ruci\'nski and Szemer\'edi~\cite{rrs} determined the minimum codegree threshold that ensures  a perfect matching in
a $k$-uniform hypergraph for all $k\ge 3$. The threshold is $n/2 - k + C$, where $C\in \{3/2, 2, 5/2, 3\}$ depends on the values of $n$ and $k$. This improved bounds given in~\cite{ko1, rrs1}.
A $k$-partite version was proved by Aharoni, Georgakopoulos and Spr\"ussel~\cite{ags}.

K\"uhn, Osthus and Treglown~\cite{KOTmatch} and independently
Khan~\cite{khan1} determined the precise minimum vertex degree threshold that forces a perfect matching in a $3$-uniform hypergraph.
(This improved on an ``asymptotically exact'' result of
H\`an, Person and Schacht~\cite{hps}.)
Recently a $3$-partite version was proved by Lo and~Markstr\"om~\cite{lomark}.
Khan~\cite{khan2} has also determined the exact minimum vertex degree threshold for $4$-uniform
hypergraphs.
(Lo and Markstr\"om~\cite{lomark2} have a proof of an approximate version of this result.)
For $k\geq 5$, the precise minimum vertex degree threshold which ensures a perfect matching in a $k$-uniform hypergraph is not
known.

The situation for $\ell$-degrees where $1 < \ell < k-1$ is also still open.
H\`an, Person and Schacht~\cite{hps} provided conditions on $\delta _{\ell} (H)$ that ensure a perfect matching in the case
when $1 \leq \ell< k/2$. These bounds were subsequently lowered by Markstr\"om and Ruci\'nski~\cite{mark}.
Recently, Alon et al.~\cite{afh} gave a connection between the minimum $\ell$-degree
that forces a perfect matching in a $k$-uniform hypergraph and the  minimum $\ell$-degree
that forces a \emph{perfect fractional matching}. As a consequence of this result they determined, asymptotically,
the minimum $\ell$-degree which forces a perfect matching in a $k$-uniform hypergraph for the following values of
$(k, \ell)$: $(4, 1)$, $(5, 1)$, $(5, 2)$, $(6, 2)$, and  $(7, 3)$.

Pikhurko~\cite{pik} showed that if $\ell \geq k/2$ and $H$ is a $k$-uniform hypergraph whose order $n$ is divisible by $k$ then
$H$ has a perfect matching provided that $\delta _{\ell} (H) \geq (1/2+o(1))\binom{n}{k-\ell}$. This result is best possible up to the $o(1)$-term
(see the constructions in $\mathcal H_{\text{ext}} (n,k)$ below).

In this paper we strengthen Pikhurko's result for $k$-uniform hypergraphs when $4$ divides $k$.
In order to state our results, we need more definitions. Fix a set $V$ of $n$ vertices. Given a partition $V$ into non-empty sets $A, B$,
let $E_{\text{odd}}(A, B)$ ($E_{\text{even}}(A, B)$) denote the family of all $k$-element subsets of $V$ that intersect $A$ in an odd (even) number of vertices.
(Notice that the ordering of the vertex classes $A,B$ is important.)
Define $\mathcal B_{n,k}(A,B)$ to be the $k$-uniform hypergraph with vertex set $V= A \cup B$ and edge set $E_{\text{odd}} (A,B)$.
%We refer to $A$ and $B$ as the \emph{vertex classes of $\mathcal B_{n,k}(A,B)$.}
Note that the complement $\overline{\mathcal B}_{n,k} (A,B)$ of $\mathcal B_{n,k} (A,B)$ has edge set $E_{\text{even}} (A,B)$.
%In the case when $|A|=\lfloor n/2 \rfloor$ and $|B|=\lceil n/2 \rceil$ we will often denote $\mathcal B_{n,k}(A,B)$ by $\mathcal B_{n,k}$ and $\overline{\mathcal B}_{n,k} (A,B)$ by $\overline{\mathcal B}_{n,k} $.

Suppose $n,k \in \mathbb N$ such that $k$ divides $n$ and $k\geq 2$. Define $\mathcal H_{\text{ext}} (n,k)$ to be the collection of the following hypergraphs. First, $\mathcal H_{\text{ext}} (n,k)$ contains all hypergraphs $\doB (A,B)$ where $|A|$ is odd. Second, if $n/k$ is odd 
then $\mathcal H_{\text{ext}} (n,k)$ also contains all hypergraphs $\dB (A,B)$ where $|A|$ is even; if $n/k$ is even then $\mathcal H_{\text{ext}} (n,k)$ also contains all hypergraphs $\dB (A,B)$ where $|A|$ is odd.

It is easy to see that no hypergraph in $\mathcal H_{\text{ext}} (n,k)$ contains a perfect matching. Indeed, first assume that $|A|$ is even and $n/k$ is odd. Since every edge of ${\mathcal B}_{n,k}(A, B)$ intersects $A$ in an odd number of vertices, one cannot cover $A$ with an odd number of disjoint odd sets. Similarly $\mathcal B_{n,k} (A, B)$ does not
contain a perfect matching if $|A|$ is odd and $n/k$ is even. Finally, if $|A|$ is odd then since every edge of $\overline{\mathcal B}_{n,k}(A,B)$ intersects $A$ in an even number of vertices, $\overline{\mathcal B}_{n,k}(A, B)$ does not contain a perfect matching.

Given $\ell \in \mathbb N$ such that $k/2 \leq  \ell \leq k-1$ define $\delta (n,k, \ell)$ to be the maximum of the minimum $\ell$-degrees among all the hypergraphs in  $\mathcal H_{\text{ext}} (n,k)$. For example, it is not hard to see that
\begin{equation} \label{eq:nk1}
\delta(n, k, k-1) = \left\{\begin{array}{ll}
{n}/{2} - k + 2 & \text{if $k/2$ is even and $n/k$ is odd}\\
{n}/{2} - k + {3}/{2} & \text{if $k$ is odd and $(n-1)/{2}$ is odd}\\
{n}/{2} - k + {1}/{2} & \text{if $k$ is odd and $(n-1)/{2}$ is even} \\
{n}/{2} - k + 1 & \text{otherwise.}
\end{array} \right.\end{equation}

The following is our main result.
\begin{thm}\label{mainthm} Suppose $r, \ell \in \mathbb N$ such that $2r\leq \ell \leq 4r-1$. Then
there exists an $n_0 \in \mathbb N$ such that the following holds. Suppose $H$ is a $4r$-uniform hypergraph on
$n \geq n_0$ vertices where $4r$ divides $n$. If
$$\delta _{\ell} (H) > \delta (n,4r,\ell)$$
then $H$ contains a perfect matching.
\end{thm}

As explained before, the minimum $\ell$-degree condition in Theorem~\ref{mainthm} is best possible.
When $k$ is divisible by $4$, Theorem~\ref{mainthm} and \eqref{eq:nk1} together give the aforementioned result of R\"odl, Ruci\'nski and Szemer\'edi \cite{rrs}.

In general, the precise value of $\delta (n,k, \ell)$ is unknown because it is not known what value of $|A|$ maximizes  the minimum $\ell$-degree of $\mathcal B _{n,k} (A, B)$ (or $\overline{\mathcal B} _{n,k} (A, B)$).  Clearly one needs to know the degree of every $\ell$-tuple of vertices from  $\mathcal B _{n,k} (A, B)$ to establish the minimum $\ell$-degree of $\mathcal B _{n,k} (A,B)$. Further, if one knows this then one can compute the  total number of edges in $\mathcal B _{n,k} (A, B)$.
However, for even $k$, it is shown in \cite[Section 3.1]{ keesud} that finding the value of $|A|$ that maximizes the number of edges in $\mathcal B _{n,k} (A,B)$ is equivalent to finding the minima of binary Krawtchouk polynomials, which is an open problem. Thus, this would suggest that calculating $\delta (n,k, \ell)$ is likely a challenging task.

In the appendix we give a tight upper bound on $\delta (n,4,2)$, which together with Theorem~\ref{mainthm} gives the minimum $2$-degree threshold
that forces a perfect matching in a $4$-uniform hypergraph. This result was recently independently proven by Czygrinow and Kamat \cite{czy}.

\begin{thm}\label{4thm} There exists an $n_0\in \mathbb N$ such that the following holds. Suppose that $H$ is a $4$-uniform
hypergraph on $n \geq n_0$ vertices where $n$ is divisible by $4$. If
$$\delta _2 (H) >\frac{n^2}{4}-\frac{5n}{4} - \frac{\sqrt{n-3}}{2}+\frac{3}{2}$$
then $H$ contains a perfect matching. Furthermore, this minimum degree condition is best possible.
\end{thm}
%Note that Theorem~\ref{4thm}, together with the results in \cite{khan2,rrs} characterise the minimum $\ell$-degree threshold that forces a perfect matching in a $4$-uniform hypergraph for all $1 \leq \ell \leq 3$.
Note that Theorem~\ref{4thm}, together with the results of R\"odl, Ruci\'nski and Szemer\'edi~\cite{rrs} and Khan~\cite{khan2},
characterize the minimum $\ell$-degree threshold that forces a perfect matching in a $4$-uniform hypergraph for all
$1 \leq \ell \leq 3$.

\medskip

The overall strategy for the proof of Theorem~\ref{mainthm} is similar to that of R\"odl, Ruci\'nski and Szemer\'edi in~\cite{rrs}, which in turn is typical for proving sharp results.
Indeed, we split the argument into `extremal' and `non-extremal' cases, and use the absorbing method developed by R\"odl, Ruci\'nski and Szemer\'edi~\cite{rrs2} in the non-extremal case. However, our non-extremal case is somewhat different from \cite{rrs}.
We concentrate on the $\ell = 2r$ case and study the structure of an auxiliary graph $G(H)$,
whose vertices are all $2r$-subsets of $V(H)$, and two $2r$-sets $U, W$ are joined by an edge if and only if $U\cup W\in E(H)$.
Furthermore, we use the hypergraph removal lemma (see e.g.~\cite{gowers, rodlskokan}) and a
structural result of Keevash and Sudakov~\cite{keesud}. %(see Section~\ref{end} for precise details).

In fact, the proof of Theorem~\ref{mainthm} is such that most of the argument extends to a more general setting. For example,
we deal with the extremal case for $k$-uniform hypergraphs for all integers $k\geq 2$. Several parts of the non-extremal
case also generalize to $2r$-uniform hypergraphs (where $r \in \mathbb N$). Thus, it seems likely that our methods may be
useful in making Pikhurko's result exact for $k$-uniform hypergraphs for all $k \geq 2$.

%In view of our analysis of the extremal case we propose the following conjecture.
\begin{conj}\label{generalconj}
Suppose $k, \ell \in \mathbb N$ such that $k/2\leq \ell \leq k-1$. Then
there exists an $n_0 \in \mathbb N$ such that the following holds. Suppose $H$ is a $k$-uniform hypergraph on
$n \geq n_0$ vertices where $k$ divides $n$. If
$$\delta _{\ell} (H) > \delta (n,k,\ell)$$
then $H$ contains a perfect matching.
\end{conj}

%%%%%%%%%%%%%%%%%%%%
\section{Notation and preliminaries}\label{sec2}
\subsection{Definitions and notation}
Given a set $X$ and an integer $r\ge 2$, we write $\binom{X}{r}$ for the set of all $r$-element subsets ($r$-subsets, for short) of $X$.
Let $k, \ell \in \mathbb N$. Suppose $H=(V,E)$ is a $k$-uniform hypergraph. Let $\{v_1, \dots, v_l \}$ be an $\ell$-subset of $V(H)$. Often we will use the notation $\u{v}$, for example, to abbreviate $\{v_1\dots v_\ell\}$. When it is clear from the context we may also write $v_1\dots v_{\ell}$ (i.e. we drop the brackets). Given
$\underline{v}\in \binom{V(H)}{\ell}$,  we write $N_H(\underline{v})$ or $N (\underline{v})$ to denote the \emph{neighborhood of $\u{v}$}, that is, the family of those $(k-\ell)$-subsets of $V(H)$ which, together with $\u{v}$, form an edge in $H$. Then $|N_H(\underline{v})| = d_H(\underline{v})$. When considering $\ell$-degree together with $\ell'$-degree for some $\ell'\neq \ell$, the following proposition is very useful (the proof is a standard counting argument, which we omit).

\begin{prop} \label{prop:deg}
Let $0\le \ell \le \ell' < k$ and $H$ be a $k$-uniform hypergraph. If $\d_{\ell'}(H)\geq x\binom{n- \ell'}{k- \ell'}$ for some $0\le x\le 1$, then $\d_{\ell}(H)\geq x\binom{n- \ell}{k- \ell}$.
\end{prop}

We denote the \emph{complement of $H$} by $\overline{H}$. That is, $\overline{H} := (V(H), \binom{V(H)}{k}\setminus E(H))$. Given a set $A \subseteq V(H)$, $H[A]$ denotes the $k$-uniform subhypergraph of $H$ \emph{induced by $A$}, namely, $H[A] := (A, E(H)\cap \binom{A}{k})$. We define $H\setminus A: = H[V(H)\setminus A]$. Given $B \subseteq E(H)$, we define $H[B] := (V(H), B)$.

Let $\eps>0$. Suppose that $H$ and $H'$ are $k$-uniform hypegraphs on $n$ vertices. We say that $H$ is \emph{$\eps$-close to $H'$}, and write $H= H'\pm \eps n ^{k}$, if $H$ becomes a copy of $H'$ after adding and deleting at most $\eps n^k$ edges. More precisely, let $A\triangle B:= (A\setminus B) \cup (B\setminus A)$ denote the \emph{symmetric difference} of two sets $A$ and $B$. Then $H$ is $\eps$-close to $H'$ if there is an isomorphic copy $\tilde{H}$ of $H$ such that $V(\tilde{H}) = V(H')$ and $|E(\tilde{H})\triangle E(H')| \le \eps n^k$.

Given a graph $G$, $x \in V(G)$ and $Y \subseteq V(G)$, we denote by $d_G (x,Y)$ the number of vertices $y \in Y$ such that
$xy \in E(G)$. A bipartite graph is called \emph{balanced} if its vertex classes have equal size.

We will often write $0<a_1 \ll a_2 \ll a_3$ to mean that we can choose the constants
$a_1,a_2,a_3$ from right to left. More
precisely, there are increasing functions $f$ and $g$ such that, given
$a_3$, whenever we choose some $a_2 \leq f(a_3)$ and $a_1 \leq g(a_2)$, all
calculations needed in our proof are valid.
Hierarchies with more constants are defined in the obvious way.
Throughout the paper we omit floors and ceilings whenever this does not affect the argument.

\subsection{The extremal graphs $\mathcal B _{n,k}$ and $\mathcal B _{n,k}(t)$.}
Given a $k$-uniform hypergraph $H$ and a partition $A,B$ of $V(H)$, an edge $e$ of $H$ is called an \emph{$A^r B^{k-r}$ edge}
if $| e\cap A|= r$ and $| e\cap B| = k-r$. An $A^r B^{k-r}$ edge is called an \emph{$(A,B)$-even edge} if $r$ is even;
otherwise we call such an edge \emph{$(A,B)$-odd}. We refer to such edges as $\emph{even}$ and $\emph{odd}$
respectively when it is clear from the context what our partition of $V(H)$ is.
Two edges of $H$ have the same \emph{parity} if both are even or both are odd. As defined earlier, $E_{\text{odd}}(A, B)$ ($E_{\text{even}}(A, B)$) is the family of all $(A, B)$-odd (-even) edges.

Suppose that $n\in \mathbb N$ such that $n\geq k \geq 2$. Let $A,B$ be a partition of a set of $n$ vertices.
Recall that $\mathcal B_{n,k}(A,B)$ is the $k$-uniform hypergraph with vertex set $A \cup B$ and edge set $E_{\text{odd}} (A,B)$, and its complement $\overline{\mathcal B}_{n,k}(A, B)$ has edge set $E_{\text{even}} (A,B)$.
When $|A|=\lfloor n/2 \rfloor$ and $|B|=\lceil n/2 \rceil$, we simply denote $\mathcal B_{n,k}(A,B)$ by $\mathcal B_{n,k}$, and $\overline{\mathcal B}_{n,k}(A, B)$ by $\overline{\mathcal B}_{n,k}$.
When $|A|=\lfloor n/2 \rfloor+t$ and $|B|=\lceil n/2 \rceil-t$ for some integer $t$ such that $-\lfloor n/2 \rfloor <t < \lceil n/2 \rceil$, we may denote $\mathcal B_{n,k}(A,B)$ by $\mathcal{B}_{n,k}(t)$. We refer to $A$ and $B$ as the \emph{vertex classes of} $\mathcal B_{n,k}$ and $\mathcal{B}_{n,k}(t)$.

%Let $\eps >0$. Suppose that $H$ is a $k$-uniform hypergraph on $n$ vertices. Then we say that $H$ is \emph{in the extremal case with parameter $\eps$} if either $H$ is $\eps$-close to $\mathcal B_{n,k}$ or $H$ is $\eps$-close to $\overline{\mathcal B}_{n,k}$.

\subsection{Absorbing sets}
Following the ideas of R\"odl, Ruci\'nski and Szemer\'edi \cite{rrs2, rrs}, we define \emph{absorbing sets} as follows:
Given a $k$-uniform hypergraph $H$,  a set $S\subseteq V(H)$ is called an \emph{absorbing set for $Q\subseteq V(H)$}, if both $H[S]$ and $H[S\cup Q]$ contain perfect matchings. In this case, if the matching covering $S$ is $M$, we also say \emph{$M$ absorbs $Q$}.

When constructing our absorbing sets in Section~\ref{sectionnon}
we will use the following Chernoff bound for binomial
distributions (see e.g.~\cite[Corollary 2.3]{Janson&Luczak&Rucinski00}).
Recall that the binomial random variable with parameters $(n,p)$ is the sum
of $n$ independent Bernoulli variables, each taking value $1$ with probability $p$
or $0$ with probability $1-p$.

\begin{prop}\label{chernoff}
Suppose $X$ has binomial distribution and $0<a<3/2$. Then
$\mathbb{P}(|X - \mathbb{E}X| \ge a\mathbb{E}X) \le 2 e^{-\frac{a^2}{3}\mathbb{E}X}$.
\end{prop}

\subsection{Two structural results for hypergraphs}
In Section~\ref{seccy} we will show that if our hypergraph $H$ does not contain a certain type of absorbing set  then $H$
is in the extremal case. To deduce this, we will obtain structural information about two auxiliary (hyper)graphs. This
 will in turn provide structural information about $H$. The following two powerful results will be required for this.
\begin{thm}[Hypergraph Removal Lemma ~\cite{gowers, rodlskokan}]
\label{thm:RS}
Let $\gamma >0$ and $k, t \in \mathbb N$ such that $2\leq k \leq t$. Given any $k$-uniform hypergraph $F$ on $t$ vertices,
there exists $\a= \a (F,\gamma ) >0$ and $n_0=n_0 (F ,\gamma) \in \mathbb N$ such that the following holds.
Suppose $H$ is a $k$-uniform hypegraph on $n \geq n_0$ vertices such that $H$ contains at most $\a n^t$ copies of $F$. Then $H$
can be made $F$-free by deleting at most $\gamma n^k$ edges.
\end{thm}
Given $r \in \mathbb N$,
let $\mathcal C ^{2r} _3$ denote the \emph{expanded $2r$-uniform triangle}. That is, $\mathcal C^{2r} _3$ consists of
three disjoint sets $P_1, P_2, P_3$ of vertices of size $r$, and the edges $P_1 \cup P_2$, $P_2 \cup P_3$, $P_3 \cup P_1$.
Keevash and Sudakov~\cite{keesud} used the following theorem  to prove a conjecture of Frankl~\cite{frank} concerning the Tur\'an number
of $\mathcal C^{2r} _3$.
\begin{thm}[\cite{keesud}]
\label{thm:Si}
For every $\gamma>0$ and $r \in \mathbb N$, there exists $\b=\b(\gamma, r)>0$ such that if $H$ is a $\mathcal C^{2r} _3$-free $2r$-uniform hypergraph on $n$ vertices
with
$$e(H) >\left(\frac12 - \b\right)\binom{n}{2r},$$
then $H= \mathcal B_{n,2r} \pm \gamma n^{2r}$.
\end{thm}

\section{Proof of Theorem~\ref{mainthm}}
Most of the  paper is devoted to the proof of the following two results: we will prove Theorem~\ref{nonexthm1} in Section~\ref{sectionnon} and Theorem~\ref{extthm} in Section~\ref{sec:ext}.

\begin{thm}\label{nonexthm1}
Let $\eps >0$ and $r, \ell \in \mathbb N$ such that $2r \leq \ell \leq 4r-1$.
Then there exist $\a, \xi >0$ and $n_0 \in \mathbb N$ such that the following holds.
Suppose that $H$ is a $4r$-uniform hypergraph on $n \geq n_0$ vertices where $4r$ divides $n$. If
$$\delta _{\ell} (H) \geq \left( \frac{1}{2}-\a \right) \binom{n-\ell}{4r-\ell}$$
then  $H$ is $\eps$-close to $\mathcal B_{n,4r}$ or $\overline{\mathcal B}_{n,4r}$, or $H$ contains
a matching $M$ of size $|M| \le \xi n/(4r)$ that absorbs any set $W\subseteq V(H) \setminus V(M)$ such that $|W| \in 4r\mathbb{N}$ with $|W| \le \xi^2 n$.
\end{thm}
Notice that the minimum $\ell$-degree condition in Theorem~\ref{nonexthm1} is weaker than that in Theorem~\ref{mainthm}.
Theorem~\ref{nonexthm1} says that either $H$ contains a reasonably small absorbing set which can absorb any small set of vertices or $H$ is `close' to $\mathcal B_{n,4r}$ or $\overline{\mathcal B}_{n,4r}$. The next result shows that in the latter, `extremal case', $H$
contains a perfect matching.

\begin{thm}\label{extthm}
Given $1\le \ell \le k-1$, there exist $\eps > 0$ and $n_0 \in \mathbb{N}$ such that the following holds. Suppose that $H$ is a $k$-uniform hypergraph on $n\ge n_0$ vertices such that $n$ is divisible by $k$. If  $\delta_{\ell} (H) >  \d(n, k, \ell)$ and $H$ is $\eps$-close to $\mathcal B_{n,k}$ or $\overline{\mathcal B}_{n,k}$, then $H$ contains a perfect matching.
\end{thm}

The following result of Markstr\"om and Ruci\'nski~\cite{mark} is needed in the `non-extremal' case.
\begin{thm}[Lemma 2 in~\cite{mark}]\label{marklemma}
For each integer $k\ge 3$, every $1 \leq \ell \leq k-2$ and every $\gamma >0$ there exists an  $n_0 \in \mathbb N$
such that the following holds.
Suppose that $H$ is a $k$-uniform hypergraph on $n\geq n_0$ vertices such that
$$\delta _{\ell} (H) \geq \left(\frac{k-\ell}{k} - \frac{1}{k^{(k-\ell)}}+\gamma \right)\binom{n-\ell}{k-\ell}.$$
Then $H$ contains a matching covering all but at most $\sqrt{n}$ vertices.
\end{thm}
In~\cite{mark}, Markstr\"om and Ruci\'nski only stated Theorem~\ref{marklemma} for $1\leq \ell <k/2$. In fact, their proof works for all values of $\ell$ such that $1 \leq \ell \leq k-2$. In the case when $\ell = k-1$, we need a result of R\"odl, Ruci\'nski and Szemer\'edi~\cite[Fact 2.1]{rrs}: if $\delta_{k-1}(H)\ge n/k$, then $H$ contains a matching covering all but at most $k^2$ vertices in $H$.

We now show that, to prove Theorem~\ref{mainthm}, it suffices to prove Theorems~\ref{nonexthm1} and~\ref{extthm}.
\medskip

\noindent
{\bf Proof of Theorem~\ref{mainthm}.}
Let $\varepsilon$ be as in Theorem~\ref{extthm} and $\a, \xi$ be as in Theorem~\ref{nonexthm1}.
That is,
$$0<\a , \xi \ll \eps \ll 1/r.$$
Assume that $2r\le \ell \le 4r -1$. Consider any sufficiently large $4r$-uniform hypergraph $H$ on $n$ vertices such that $4r$ divides $n$ and
$$\delta _\ell (H) >\delta (n,4r,\ell) .$$

For any $k\ge 2$, it is clear that $\delta_{k-1}(\mathcal{B}_{n, k})\ge n/2 - (k-1)$. Thus, by Proposition~\ref{prop:deg}, $\delta _{\ell} (\mathcal B_{n,4r}) \geq (1/2 -\alpha)\binom{n-\ell}{4r-\ell}$. Consequently $\delta _{\ell} (H) \geq (1/2- \a) \binom{n-\ell}{4r-\ell}$. Theorem~\ref{nonexthm1} implies that
either $H$ is $\eps$-close to $\mathcal B_{n,k}$ or $\overline{\mathcal B}_{n,k}$ or $H$ contains
a matching $M$ of size $|M| \le \xi n/(4r)$ that absorbs any set $W\subseteq V(H) \setminus V(M)$ such that $|W| \in 4r\mathbb{N}$ with $|W| \le \xi^2 n$.
In the former case Theorem~\ref{extthm} implies that $H$ contains a perfect matching. In the latter case set $H':=H\backslash V(M)$ and $n':=|V(H')|$.
Since $\ell\ge 2r$, $\alpha, \xi \ll 1/r$ and $n$ is sufficiently large,
$$\delta _{\ell} (H') \geq \delta _{\ell} (H) - |V(M)|\binom{n}{4r-\ell-1} \geq
\left(\frac{4r-\ell}{4r} - \frac{1}{(4r)^{(4r-\ell)}}+\a \right)\binom{n'-\ell}{4r-\ell}.$$
Hence, if $\ell \leq 4r-2$, Theorem~\ref{marklemma} implies that $H'$ contains a matching $M'$ covering all but at most $\sqrt{n'}$ vertices in $H'$. If $\ell=4r-1$, then since $\delta _{\ell} (H') \geq n'/(4r)$, Fact 2.1 from~\cite{rrs} implies that $H'$  contains a matching $M'$ covering all but at most $(4r)^2$ vertices in $H'$.
In both cases set $W:=V(H')\backslash V(M')$. Then $|W|\leq \sqrt{n'} \le \xi^2 n$. By definition of $M$, there is a matching $M''$ in $H$ which covers
$V(M)\cup W$.  Thus, $M' \cup M''$ is a perfect matching of $H$, as desired.
\endproof

%%%%%%%%%%%%%%%%%%%%%%%%
\section{The Extremal Case}
\label{sec:ext}

In this section we prove Theorem~\ref{extthm}: for sufficiently small $\eps>0$ and sufficiently large $n \in k\mathbb{N}$, any $k$-uniform $n$-vertex hypergraph $H$ with $\delta_{\ell} (H) >  \d(n, k, \ell)$ and which is $\eps$-close to $\dB$ or $\doB$ contains a perfect matching.
%In fact, we prove a stronger statement, Theorem~\ref{thm:ext} below.
Recall that $\delta (n,k, \ell)$ is the maximum of the minimum $\ell$-degrees among all the hypergraphs in  $\mathcal H_{\text{ext}} (n,k)$, and $\mathcal H_{\text{ext}} (n,k)$ contains all hypergraphs $\doB (A,B)$ with $|A|$ odd, and all
hypergraphs $\dB (A,B)$ where $n/k$ is odd and $|A|$ is even, and where $n/k$ is even and $|A|$ is odd.

Given two $k$-uniform hypergraphs $H$ and $H'$ on $n$ vertices, we say $H$ \emph{$\eps$-contains $H'$} if, after adding at most $\eps n^k$ edges to $H$, the resulting hypergraph contains a copy of $H'$. More precisely, $H$ $\eps$-contains $H'$ if there is an isomorphic copy $\tilde{H}$ of $H$ such that $V(\tilde{H}) = V(H')$ and $|E(H')\setminus E(\tilde{H})| \le \eps n^k$. Trivially if $H$ is $\eps$-close to $H'$, then $H$ $\eps$-contains $H'$.

The following theorem thus implies Theorem~\ref{extthm}.
\begin{thm}\label{thm:ext}
Given $1\le \ell \le k-1$, there exist $\eps > 0$ and $n_0 \in \mathbb{N}$ such that the following holds. Suppose that $H$ is a $k$-uniform hypergraph on $n\ge n_0$ vertices such that $n$ is divisible by $k$. Then $H$ contains a perfect matching if the following holds.
\begin{itemize}
\item $\delta_{\ell} (H) >  \d(n, k, \ell)$;
\item $H$ $\eps$-contains $\dB$ or $\doB$.
\end{itemize}
\end{thm}

%\subsection{Another structural extremal case result}
Furthermore, by modifying the proof of Theorem~\ref{thm:ext} slightly one can obtain another structural extremal case result (we omit its proof).
\begin{thm}\label{thm:ext2}
Given an integer $k \geq 2$, there exist $\eps > 0$ and $n_0 \in \mathbb{N}$ such that the following holds. Suppose that $H$ is a $k$-uniform hypergraph on $n\ge n_0$ vertices such that $n$ is divisible by $k$. Then $H$ contains a perfect matching if the following holds.
\begin{description}
\item[(i)] $\d_1(H)\ge (\frac12 - \eps) \binom{n-1}{k-1}$;
\item[(ii)] Under any partition $A, B$ of $V(H)$,
there always exist at least one $(A, B)$-even edge and at least one $(A, B)$-odd edge;
\item[(iii)] $H$ $\eps$-contains $\dB$ or $\doB$.
\end{description}
\end{thm}

The rest of this section is devoted to the proof of Theorem~\ref{thm:ext}.

\subsection{Preliminaries and proof outline}

Given a set $A$, we denote by $K^k(A)$ the complete $k$-uniform hypergraph on $A$ (the superscript $k$ is often omitted).
Given integers $0\le r\le k$ and two disjoint sets $A$ and $B$, let $K^k_r(A, B)$ or simply $K_r(A, B)$ denote the $k$-uniform hypergraph on $A\cup B$ whose edges are all $k$-sets intersecting $A$ with precisely $r$ vertices.

Let $H, H'$ be two $k$-uniform hypergraphs on the same vertex set $V$. Let $H' \setminus H := (V, E(H')\setminus E(H))$. Suppose that $0\le \alpha\le 1$ and $|V|= n$. A vertex $v\in V$ is called \emph{$\alpha$-good in $H$} (otherwise \emph{$\alpha$-bad}) with respect to $H'$ if $d_{{H'}\setminus H}(v) \le \alpha n^{k-1}$. Sometimes we also say that $v$ is \emph{$\alpha$-good (in $H$) with respect to $E(H')$}.

We use the following result \cite[Fact 4.1]{rrs} and include a proof for completeness.

\begin{lemma}
\label{goodlem:r}
Let $k, r\in \mathbb N$ such that $k\ge 2$ and $ r\le k$. Let $0< \alpha < \frac{1}{ k (2k(k-1))^{k-1} }$. Suppose that $H$ is a $k$-uniform hypergraph on $V= A\cup B$ such that $|A|= tr$, $|B|= t(k-r)$ for some integer $t\ge 2(k-1)$, and every vertex of $H$ is $\alpha$-good with respect to $K^k_r(A, B)$. Then $H$ contains a perfect matching.
\end{lemma}

\proof Let $M$ be a largest matching of $H$ consisting of only $A^r B^{k-r}$ edges. Set $m := |M|$
and $n:=|V|=tk$. We claim that $m= t$, namely, $M$ is a perfect matching of $H$. Suppose $m< t$ instead.
Let $A_0 := A\setminus V(M)$ and $B_0 :=B\setminus V(M)$. Then $|A_0|= (t - m) r\ge r$ and $|B_0|= (t - m)(k - r) \ge k-r$. The maximality of $M$ implies that there are no $A_0^r B_0^{k-r}$ edges. Fix $v\in A_0$. Since $v$ is $\alpha$-good with respect to $K^k_r(A, B)$, it follows that $\binom{|A_0| -1}{r-1} \binom{|B_0|}{k-r} \le \alpha n^{k-1}$, which implies that
\[
\left(\frac{|A_0|}{r} \right)^{r-1} \left(\frac{|B_0|}{k-r}\right)^{k-r} \le \alpha n^{k-1}
\]
and thus, $(t - m)^{k-1} \le \alpha (t k)^{k-1}$. Since $\alpha < 1/ (2k)^{k-1}$, this implies that $t - m\le t/2$ or $m\ge t/2$.

Fix a $k$-set $S= \{v_1, v_2, \dots, v_k\}$ with $v_1, \dots, v_r\in A_0$ and $v_{r+1}, \dots, v_k\in B_0$.
Given a vertex $v \in V$, we call a collection $e_1, \dots , e_{k-1}$ of $k-1$ distinct edges \emph{feasible for $v$}
if every $k$-set $T$ with $v \in T$, $|T\cap e_i|=1$ for all $1\leq i \leq k-1$ and $|T \cap A|=r$ is an edge of $H$.
We claim that there are $k-1$ (distinct) edges $e_1, \dots, e_{k-1}$ of $M$ that are feasible for all the vertices of $S$.
This contradicts the maximality of $M$ since it is easy to see that $\bigcup_{i=1}^{k-1} e_i\cup S$ contains $k$ disjoint
$A^r B^{k-r}$ edges of $H$.

To find $k-1$ feasible edges for all the vertices of $S$, we
consider all $(k-1)$-tuples of $M$.
%namely, all $k-1$ distinct edges from $M$.
There are $\binom{|M|}{k-1}\ge \binom{t/2}{k-1}$ $(k-1)$-tuples of $M$. Since each $v_i$ is $\alpha$-good, at most $\a n^{k-1}$ $(k-1)$-sets that are neighbors of $v$ in $K^k_r(A,B)$ are not neighbors of $v_i$ in $H$.  Thus at most $\a n^{k-1}$ $(k-1)$-tuples of $M$ are not feasible for $v_i$. In
total, at most $k \a n^{k-1}$ $(k-1)$-tuples of $M$ are not feasible for at least one vertex of $S$. Since $t/2 \ge k-1$ and $\alpha < \frac{1}{k (2k(k-1))^{k-1}}$, we have $\binom{t/2}{k-1} \ge (\frac{t}{2(k-1)})^{k-1} > k \alpha n^{k-1}$. Hence there always exists a $(k-1)$-tuple of $M$ feasible for all the vertices of $S$.
\endproof

To derive Corollary~\ref{cor:good}, we also need a simple claim.
\begin{claim}
\label{clm:HU}
Let $H$ and $H'$ be two $k$-uniform hypergraphs on an $n$-vertex set $V$. Suppose that $\a >0$ and $v$ is $\a$-good in $H$ with respect to $H'$. Let $H''$ be a subgraph of $H'$ on $U\subset V$ such that $v\in U$ and $|U|\ge cn$ for some $c>0$. Then $v$ is $\a'$-good in $H[U]$ with respect to $H''$, where $\a':= \a/ c^{k-1}$.
\end{claim}
\begin{proof}
This follows from
\[
d_{H''\setminus H[U]}(v) \le d_{H' \setminus H}(v) \le \a n^{k-1} = \a' (cn)^{k-1}\le \a' |U|^{k-1}. \qedhere
\]
\end{proof}

\begin{cor}\label{cor:good}
Given an even integer $k\ge 2$, there exist $\alpha > 0$ and $n_0\in \mathbb{N}$ such that the following holds for all $n\ge n_0$ with $n\in 2k \mathbb{N}$. Suppose that $H$ is an $n$-vertex $k$-uniform hypergraph with a partition $A, B$ of $V(H)$ such that $|A|= |B| = n/2$. If every vertex of $H$ is $\alpha$-good with respect to $\dB (A, B)$,
 then $H$ contains a perfect matching.

Furthermore, if $k/2$ is odd, then $n\in 2k \mathbb{N}$ can be weakened to $n\in k \mathbb{N}$.
\end{cor}

\begin{proof}
First assume that $n\in 2k \mathbb{N}$. Then $|A|= |B|$ is divisible by $k$.
We arbitrarily partition $A$ into two subsets $A_1$ of size $|A|/k$ and $A_2$ of size $|A| (k-1)/k$, and partition $B$ into two subsets $B_1$ of size $|B|(k-1)/k$ and $B_2$ of size $|B|/k$. Let $H_i = H[A_i\cup B_i]$ for $i=1, 2$.
Since all the vertices of $H$ are $\alpha$-good with respect to $\dB (A, B)$, by Claim~\ref{clm:HU}, all the vertices in $A_1\cup B_1$ are $\a'$-good in $H_1$ with respect to $K_1(A_1, B_1)$, where $\a' := 2^{k-1} \a$.
Similarly, every vertex in $A_2\cup B_2$ is $\a'$-good in $H_2$ with respect to $K_1(A_2, B_2)$.
 As $\a'\ll 1/k$, we can apply Lemma~\ref{goodlem:r} to $H_1$ and $H_2$ obtaining a perfect matching $M_1$ of $H_1$ and a perfect matching $M_2$ of $H_2$. Thus $M_1\cup M_2$ is a perfect matching of $H$.

Second assume that $k/2$ is odd and $n\in k \mathbb{N}$. Then $|A|= |B|$ is divisible by $k/2$. Since every vertex of $H$ is $\alpha$-good with respect to $K_{k/2} (A, B)$,
we can apply Lemma~\ref{goodlem:r} with $r=k/2$ obtaining a perfect matching of $H$.
\end{proof}

Now we give an outline of our proof of Theorem~\ref{thm:ext}.
\begin{description}
\item[Step 1] Since $H$ $\eps$-contains $\dB$ (or $\doB$), all but at most $\eps_1 n$ vertices in $H$ are $\eps_2$-good with respect to $\dB$ (or $\doB$) for some $\eps \ll \eps_1\ll \eps_2$. Denote the set of $\eps_2$-bad vertices by $V_0$.
Let $A$ and $B$ denote the vertex classes of $\dB$ (or $\doB$).
We move the vertices of $V_0$ to the other side (from $A$ to $B$ or from $B$ to $A$) and denote the resulting sets by $A_1$ and $B_1$.

\item[Step 2] In some cases, we will obtain a special edge $e_0$, which is
an $(A_1, B_1)$-even edge when $H$ $\eps$-contains $\dB$ or an $(A_1, B_1)$-odd edge when $H$ $\eps$-contains $\doB$.
Note that $e_0$ may contain vertices of $V_0$.

\item[Step 3] We remove a matching $M_1$ of size $|M_1|\le \eps_1 n$ containing all the vertices in $V_0\setminus e_0$. Denote the resulting sets by $A_2$ and $B_2$.

\item[Step 4] We remove a small matching from $H[A_2\cup B_2]$ such that the resulting sets $A_3, B_3$ satisfy:
\begin{itemize}
\item If $k$ is even and $H$ $\eps$-contains $\doB$, then $|A_3|\equiv 0 \pmod{k}$.
\item If $k$ is even and $H$ $\eps$-contains $\dB$, then $|A_3|=|B_3|$. Furthermore, if $k$ is divisible by $4$, we also need $|A_3|\equiv 0 \pmod{k}$.
\item If $k$ is odd, then $|A_3|\equiv 0 \pmod{k-1}$.
\end{itemize}
In many cases the special edge $e_0$ is needed in this step.

\item[Step 5] If $e_0$ was introduced in Step~2 but not used in Step~4 and $e_0\cap V_0\neq \emptyset$, we remove a small matching containing all the vertices in $e_0\cap V_0$ while preserving the property mentioned in Step~4.

\item[Step 6] We apply Lemma~\ref{goodlem:r} or Corollary~\ref{cor:good} to $H[A_3\cup B_3]$ and find a perfect matching of $H[A_3\cup B_3]$.
\end{description}
In the next three subsections, we give details of these steps based on the three cases listed in Step~4. Full details for each step are only given when the step is needed at the first time. Note that Steps 1 and 3 are essentially the same for all the three cases but Steps 2 and 5 are not necessary in some cases.
%The edge $e_0$ serves as our parity breaker in $H[A_2\cup B_2]$. Sometimes we can tell if $e_0$ is needed from the sizes of $A_1$ and $B_1$ and do not need Step~5. However, as in Section~\ref{sec:ext2}, it is harder to predict if $e_0$ is necessary when only knowing $A_1$ and $B_1$.

Indeed, we may only apply Step 2 in the case when, after applying Step 1,
(i) $H$ $\eps$-contains  $\dB$ and $\dB (A_1,B_1) \in \mathcal H_{\text{ext}}(n,k)$ or;
(ii) $H$ $\eps$-contains  $\doB$ and $\doB (A_1,B_1) \in \mathcal H_{\text{ext}}(n,k)$.
In these cases, we will need to use the condition that $\delta _{\ell} (H) >\delta (n,k,\ell)$ to ensure
$H$ contains our desired edge $e_0$. This is the only place in the proof of Theorem~\ref{thm:ext} (and in fact, the
only part of the proof of Theorem~\ref{mainthm})
where we use the full force of our minimum $\ell$-degree condition.

The edge $e_0$ acts as a ``parity-breaker'', helping us to construct our desired perfect matching.
However, if $H$ does not satisfy (i) or (ii) then no parity-breaking edge is required, and so we do not need Step 2.

\subsection{$k$ is even and $H$ $\eps$-contains $\doB$}
\label{sec:ext1}
In this subsection, we prove Theorem~\ref{thm:ext} under the assumption that $k$ is even and $H$ $\eps$-contains $\doB$,
where $0< \eps \ll 1/k$. Define $\eps_1 := {k}^{\frac{1}{2}} \eps^{\frac23}$ and $\eps_2 := {k}^{\frac12} \eps^{\frac13}$.  Let $H$ be a $k$-uniform hypergraph on an $n$-vertex set $V$ for sufficiently large $n\in k\mathbb{N}$. Note that $n$ is even because $k$ is even. Suppose that $H$ $\eps$-contains $\doB$, namely, there exists a partition $A,B$ of $V$ such that $|A| = |B| = n/2 $, $\doB = (V, E_{\text{even}}(A, B))$ and $|E_{\text{even}}(A, B)\setminus E(H)| \le \eps n^k$.

\medskip

\noindent \textbf{Step 1:} Recall that a vertex $v \in V(H)$ is \emph{$\eps_2$-bad} with respect to $\doB$ if $d_{{\doB}\setminus H} (v) > \eps_2 n^{k-1}$. In other words, if $v$ is $\eps_2$-good then all but at most $\eps_2 n^{k-1}$ of the $(A,B)$-even edges that contain $v$ belong to $H$. We observe that at most $\eps _1 n$ vertices in $H$ are $\eps _2$-bad. Otherwise
\[
k |E(\doB)\setminus E(H)| = \sum_{v\in V} |N_{\doB}(v)\setminus N_H(v)| > \eps_2 n^{k-1} \eps_1 n = k \eps n^k,
\]
contradicting the assumption that $|E_{\text{even}}(A, B)\setminus E(H)| \le \eps n^k$.

Let $A_{0}$ and $B_{0}$ denote the sets of $\eps _2$-bad vertices in $A$ and in $B$, respectively, and
set $V_0:= A_{0} \cup B_{0}$. Then $|A_{0}| + |B_{0}| = |V_0| \leq \eps _1 n$.
Notice that $\d_1(H)\ge (\frac12 - \eps) \binom{n-1}{k-1}$ by Proposition~\ref{prop:deg}. Consider $v\in V_{0}$. We know that $d_{\doB}(v) \le (\frac12 + \eps) \binom{n-1}{k-1}$. Since $ d_{\doB\setminus H} (v) > \eps_2 n^{k-1}$, it follows that
\begin{equation}
\label{eq:dv0}
d_{H\setminus \doB} (v)  \ge \left(\tfrac12 - \eps \right) \binom{n-1}{k-1} - (d_{\doB}(v) - \eps_2 n^{k-1}) \ge \eps_2 n^{k-1} - 2\eps \binom{n-1}{k-1}\ge \frac{\eps_2}{2} n^{k-1}.
\end{equation}
In other words, $v$ lies in at least $\eps_2 n^{k-1}/2$ $(A, B)$-odd edges in $H$.

Define $A_1:=(A\setminus A_{0}) \cup B_{0}$ and $B_1:=(B\setminus B_{0}) \cup A_{0}$. Then $A_1$, $B_1$ is a partition of $V(H)$ with $|A_1|, |B_1| \ge (1/2 - \eps _1 )n$.
%Without loss of generality we may assume that %$|A_1|\geq |B_1|$.

We now separate cases based on the parity of $|A_1|$.

First assume that $|A_1|$ is even. Then $\doB (A_1,B_1) \not \in \mathcal H_{\text{ext}} (n,k)$. Thus,
we do not need Step 2 (and therefore Step 5) in this case.

\noindent \textbf{Step 3:} We remove a matching $M_1$ from $H$ such that
\begin{itemize}
\item $|M_1|= |V_0| \leq \eps_1 n$;
\item each edge of $M_1$ contains exactly one vertex of $V_0$;
\item all the edges of $M_1$ are $(A _1, B_1)$-even.
\end{itemize}
To find $M_1$, we consider the vertices of $V_0$ in an arbitrary order and apply the following simple claim repeatedly.
\begin{claim}
\label{clm:Uv}
Let $k \geq 2$ be an integer and $\alpha _1 , \alpha _2 $ be constants such that
 $\a_2 > \a_1/(k-2)! \ge 0$ (here $0!:=1$). Let $H$ be a $k$-uniform hypergraph on $n$ vertices such that $d_H(v)\ge \a_2 n^{k-1}$ and $|U|\le \a_1 n$ for some $U\subset V(H)$ with $v\not\in U$. Then $v$ lies in an edge disjoint from $U$.
\end{claim}
\begin{proof} There are at most $\a_1 n \binom{n-2}{k-2} \le \frac{\a_1}{(k-2)!} n^{k-1}$ edges of $H$ containing $v$ and at least
one vertex from $U$. Since $\a_2 > \a_1/(k-2)!$, there exists an edge containing $v$ and no vertex of $U$.
\end{proof}
Suppose that we have found $i$ edges in $M_1$ and consider the next vertex $v\in V_0$. Then $|V_0 \cup V(M_1)|\le k \eps_1 n$. Because of \eqref{eq:dv0} and  $\eps_1\ll \eps_2 $, we can apply Claim~\ref{clm:Uv} with $U= (V_0\setminus \{v\}) \cup V(M_1)$ to find an $(A, B)$-odd edge containing $v$ but no other vertex of $V_0$ and which is disjoint from the existing edges of $M_1$. By the definition of $A_1, B_1$, any $(A,B)$-odd edge containing $v$ and no other vertex of $V_0$ is an $(A_1,B_1)$-even edge. We thus add this edge to $M_1$. At the end of this process, let $A_2:=A_1\setminus V(M_1)$ and $B_2:=B_1 \setminus V(M_1)$.

\medskip
\noindent \textbf{Step 4:} Since $|A_1| $ is even, the third property of $M_1$ implies that $s := |A_2| \pmod{k}$ is also even. If $s\neq 0$, we remove an $A_2^s B_2^{k-s}$ edge $e_2$. Such an edge exists because all the vertices in $A_2\cup B_2$ are $\eps_2$-good with respect to $\doB$. More precisely, since $A_2\subseteq A$, $B_2\subseteq B$, and $|A_2|, |B_2|\ge (\frac12 - (k+1)\eps_1)n$, Claim~\ref{clm:HU} implies that all the vertices in $A_2\cup B_2$ are $2\eps_2$-good with respect to $K_s(A_2, B_2)$. As $\eps_2 \ll 1/k$ and consequently
\[
2\eps_2 n^{k-1} < \binom{(\frac12 - (k+1)\eps_1)n - 1}{s-1} \binom{(\frac12 - (k+1)\eps_1)n}{k-s},
\]
there exists an $A_2^s B_2^{k-s}$ edge containing \emph{any} vertex in $A_2$.

Let $A_3:= A_2\setminus e_2$ and $B_3:= B_2\setminus e_2$. Then $|A_3| \equiv 0 \pmod k$. Since $|A_3| + |B_3| \equiv |A| + |B| \equiv 0 \pmod k$, we have $|B_3| \equiv 0 \pmod k$.

\medskip
\noindent \textbf{Step 6:}
Since $|A_3|\ge (1/2- 2k\eps_1)n \ge n/3$, by Claim~\ref{clm:HU}, all the vertices of $A_3$ are $(3^{k-1} \eps_2)$-good in $H[A_3]$ with respect to $\doB[A_3] = K^k(A_3)$, the complete $k$-uniform hypergraph on $A_3$. As $\eps_2\ll 1/k$, by Lemma~\ref{goodlem:r}  (with $r=k$),\COMMENT{When $r=k$, instead of Lemma~\ref{goodlem:r}, we could use the  (stronger) result of Daykin and H\"{a}ggkvist \cite{dayhag} mentioned in Section~1.} there is a perfect matching $M_3$ of $H[A_3]$. Similarly we can find a perfect matching $M'_3$ of $H[B_3]$ (note that $\doB[B_3] = K^k(B_3)$ because $k$ is even). The union $M_1\cup \{e_2\} \cup M_3\cup M'_3$ is the desired perfect matching of $H$.

\medskip

Now assume that $|A_1|$ is odd. In this case we need Step~2 (but not Step~5). Note that $\doB (A_1,B_1) \in \mathcal H_{\text{ext}} (n,k)$ since $|A_1|$ is odd. As $\delta _{\ell} (H) > \delta (n,k,\ell) \geq \delta _\ell (\doB (A_1,B_1))$, we can find an $(A_1, B_1)$-odd edge $e_0$. We apply Step~3 as before though now we require that $M_1$ is chosen to be disjoint from $e_0$. In particular, this means
$M_1$ is chosen to cover $V_0 \backslash e_0$. After Step~3, we let $A'_2:=A_2 \setminus e_0$ and $B'_2:=B_2 \setminus e_0$. Then $s:= |A'_2| \pmod k$ is even. The rest of the argument
 is the same as in the case when $|A_1|$ is even.
\endproof

\subsection{$k$ is even and $H$ $\eps$-contains $\dB$}
\label{sec:ext2}

Assume that $k$ is even, and $n$ is sufficiently large and divisible by $k$ (thus $n$ is also even). Recall that $\dB$ is the $k$-uniform hypergraph whose vertex set is partitioned into $A\cup B$ such that $|A|=|B|= n/2$ and whose edge set $E_{\text{odd}}(A, B)$ consists of all $k$-sets that intersect $A$ in an odd number of vertices. Suppose that $H$ is a $k$-uniform hypergraph on $n$ vertices such that $H$ $\eps$-contains $\dB$, namely, $|E_{\text{odd}}(A, B)\setminus E(H)| \le \eps n^k$.

\textbf{Step~1} is the same as in Section~\ref{sec:ext1}, except for replacing $\doB$ by $\dB$.
Therefore again $A_0$ and $B_0$ denote the sets of $\eps_2$-bad vertices in $A$ and $B$ respectively and $V_0:=A_0 \cup B_0$,
$A_1:=(A\backslash A_0)\cup B_0$ and $B_1:=(B\backslash B_0)\cup A_0$.

If $\dB (A_1,B_1)  \in \mathcal H_{\text{ext}} (n,k)$ then as $\delta _{\ell} (H) > \delta _{\ell} ( \dB (A_1,B_1) )$,
we can apply  \textbf{Step 2}. That is,  $H$
contains an $(A_1,B_1)$-even edge $e_0$.
Then $r_0 := |e_0\cap A _1|$ is even.
If $\dB (A_1,B_1)  \not \in \mathcal H_{\text{ext}} (n,k)$ then we do not apply Step 2. (So in what follows, we take
$e_0 = \emptyset$ in this case.)

In \textbf{Step~3}, we remove a matching $M_1$ such that
\begin{itemize}
\item $|M_1|= |V_0\backslash e_0| \leq \eps_1 n$;
\item each edge of $M_1$ contains exactly one vertex of $V_0\backslash e_0 $;
\item all the edges of $M_1$ are $(A _1, B_1)$-\emph{odd} and are disjoint from $e_0$.
\end{itemize}
Further, in the case when $\dB (A_1,B_1) \not \in \mathcal H_{\text{ext}} (n,k)$ we add at most $3$ extra $(A_1,B_1)$-odd edges to $M_1$
to ensure that $M_1$ is a matching with $|M_1|$ divisible by $4$.
Set $A_2 := A_1\setminus V(M_1)$ and $B_2 := B_1\setminus V(M_1)$. Without loss of generality, assume that $|A_2|\ge |B_2|$. Let $d := |A_2| - |B_2|$. Then $d$ is even because $|A_2| + |B_2|$ is even. We also know that $d\le k|M_1|+2|V_0|\le (k+2)\eps_1 n+3k$.
We now separate cases based on the parity of $k/2$.

\subsubsection{$k/2$ is even}

\noindent \textbf{Step 4:} We remove a matching $M_2$ that consists of $d/2$ $A_2^{k/2 + 1} B_2^{k/2 - 1}$ edges
that are disjoint from $M_1$ and $e_0$
(note that $k/2 + 1$ and $k/2 - 1$ are odd). In a similar way to Step~4 of Section~\ref{sec:ext1}, these edges exist because all the vertices in $(A_2\cup B_2)\setminus e_0$ are $\eps_2$-good with respect to $\dB$.
The resulting sets $A_3:= A_2 \setminus V(M_2)$ and $ B_3:= B_2 \setminus V(M_2)$ thus have the same size
\[
|A_2| - \frac{d}{2} \left(\frac{k}{2} + 1 \right) = |B_2| - \frac{d}{2} \left(\frac{k}{2} - 1 \right) .
\]
Let $s := |A_3| = |B_3| \pmod k$. Since $|A_3| + |B_3| \equiv 0 \pmod{k}$, it follows that either $s=0$ or $s= k/2$.

Notice that if $\dB (A_1,B_1) \not \in \mathcal H_{\text{ext}} (n,k)$, then $s=0$. Indeed, suppose not. Then $s=k/2$ and so
$|A_3|=|B_3|=km+k/2$ for some $m \in \mathbb N$. Thus, $|A_3|+|B_3|=2km+k$. Hence,
$(|A_3|+|B_3|)/k$ is odd but $|A_3|$ is even. Since the edges in $M_1 \cup M_2$ are $(A_1,B_1)$-odd, this implies
that either $(|A_1|+|B_1|)/k=n/k$ is odd and $|A_1|$ is even or $n/k$ is even and $|A_1|$ is odd. In both cases this
implies that $\dB (A_1,B_1)  \in \mathcal H_{\text{ext}} (n,k)$, a contradiction.
\medskip

\textbf{Case 1a: $s=0$.}
If $e_0\cap V_0 = \emptyset$, then we proceed to \textbf{Step~6} directly. Since $|A_3| = |B_3| \equiv 0 \pmod{k}$ and $|A_3|, |B_3|\ge (\frac12 - 2k^2\eps_1)n$, we can apply Corollary~\ref{cor:good} obtaining a perfect matching $M_3$ of $H[A_3\cup B_3]$. Consequently $M_1\cup M_2\cup M_3$ is the desired perfect matching of $H$. (Note that this covers the case when
$\dB (A_1,B_1) \not \in \mathcal H_{\text{ext}} (n,k)$, since $s=0$ and $e_0 = \emptyset$ in this case.)

If $e_0\cap V_0 \neq \emptyset$, then we need \textbf{Step~5}, in which we remove a small matching containing all the vertices of $e_0\cap V_0$. Let $v\in e_0\cap V_0$. By a similar calculation as in \eqref{eq:dv0}, $v$ is contained in at least $\eps_2 n^{k-1}/2$ $(A,B)$-even edges.
Applying Claim~\ref{clm:Uv} with $U=V(M_1\cup M_2)\cup(e_0 \backslash v)$, we find an $(A,B)$-even edge of $H[A_3\cup B_3]$ containing $v$. Since $v$ changes `side' (from $A$ to $B_1$ or from $B$ to $A_1$), and by the choice of $U$,
 this edge is an $A_3^r B_3^{k-r}$ edge for some \emph{odd} $r$. To keep the numbers of the remaining vertices in $A_3$ and $B_3$ the same and divisible by $k$,
when we remove an $A^r _3 B^{k-r} _3$ edge $e$ containing $v$
we immediately remove an $A_3^{k-r} B_3^r$ edge disjoint from $e$ (such an edge exists because all the vertices in $(A_3\cup B_3) \setminus e_0$ are $\eps_2$-good with respect to $\dB$). Repeat this process for all the vertices in $e_0\cap V_0$. Denote by $M_3$ the set of all removed edges in this step. Then $|M_3| \le 2k$. Let $A_4:= A_3 \setminus V(M_3)$ and $ B_4:= B_3 \setminus V(M_3)$. Then $|A_4|= |B_4| \equiv 0 \pmod k$. Finally in \textbf{Step~6} we find a perfect matching $M_4$ of $H[A_4\cup B_4]$ by Corollary~\ref{cor:good}. Thus $M_1\cup M_2\cup M_3\cup M_4 $ is the desired perfect matching of $H$.

\textbf{Case 1b: $s= k/2$.} Recall that $|e_0 \cap A_1| = r_0$ for some even $r_0$. Thus, $|e_0 \cap A_3| = r_0$.
 We continue on \textbf{Step~4} as follows. If $r_0\le k/2$, then we remove $e_0$ together with $\frac{k}2 - r_0$
 disjoint $A_3^{k/2+1} B_3^{k/2-1}$ edges; otherwise we remove $e_0$
together with $r_0 - \frac{k}2$ disjoint $A_3^{k/2-1} B_3^{k/2+1}$ edges. Denote by $M_3$ the set of these removed edges.
%Then $|M_3|= 1 + | k/2 - r_0|$ is odd because $k/2$ and $r_0$ are both even.
Let $A_4:= A_3 \setminus V(M_3)$ and $ B_4:= B_3 \setminus V(M_3)$. It is easy to see that $|A_4| = |B_4| = |A_3| - (|\frac{k}2 - r_0|
 +1) \frac{k}2$.
Since $s=k/2$ and $k/2, r_0$ are even, we have $|A_4|\equiv 0 \pmod k$.  Since $e_0$ has been used, we now skip Step~5 and proceed to \textbf{Step~6}. As in Case~1a, we find a perfect matching $M_4$ of $H[A_4\cup B_4]$ by Corollary~\ref{cor:good}. Consequently $M_1\cup M_2\cup M_3\cup M_4 $ is the desired perfect matching of $H$.

\subsubsection{$k/2$ is odd}
Recall that $d := |A_2| - |B_2|\ge 0$ is even. We will separate cases based on the parity of $d/2$.
Firstly though, notice that if $\dB (A_1,B_1) \not \in \mathcal H_{\text{ext}} (n,k)$ then $d$ is divisible by $4$.
Indeed, suppose instead that $d \equiv 2 \pmod 4$. First consider the case when $|A_2|+|B_2|$ is divisible by $4$.
Since $|M_1|$ is divisible by $4$, this implies that $|A_1|+|B_1|=n$ is divisible by $4$. But since $k$ is not divisible by $4$, this implies that $n/k$ is even.
Further, since $d \equiv 2 \pmod 4$, we derive that $|A_2|$ is odd. Since $|A_1\setminus A_2|$ is even, this implies that $|A_1|$ is odd. Therefore $\dB (A_1,B_1)  \in \mathcal H_{\text{ext}} (n,k)$, a contradiction. Second assume that $|A_2|+|B_2|\equiv 2 \pmod 4$ (recall that $|A_2|+|B_2|$ is even). Since $|M_1|$ is divisible by $4$, this implies that $n\equiv 2 \pmod 4$. As $k$ is even, this implies that $n/k$ is odd. So as $d \equiv 2 \pmod 4$, we derive that $|A_2|$ is even, and consequently $|A_1|$ is even. Therefore $\dB (A_1,B_1) \in \mathcal H_{\text{ext}} (n,k)$, a contradiction.

\textbf{Case 2a: $4$ divides $d$.} In \textbf{Step 4}, we remove $d/4$ disjoint
$A_2^{k/2 +2} B_2^{k/2 -2}$ edges (these edges exists because $k/2+2$ is odd and all the vertices $(A_2\cup B_2)\setminus e_0$ are $\eps_2$-good with respect to $\dB$). Denote by $M_2$ the set of these edges. Let $A_3:= A_2 \setminus V(M_2)$ and $ B_3:= B_2 \setminus V(M_2)$. Then
\[
|A_3| = |A_2| - \frac{d}{4} \left(\frac{k}{2} + 2 \right) = |B_2| -  \frac{d}{4} \left(\frac{k}{2} - 2 \right) = |B_3|.
\]
If $e_0\cap V_0 = \emptyset$, then we proceed to \textbf{Step~6}.
Claim~\ref{clm:HU} implies that all the vertices in $H[A_3\cup B_3]$ are $2\eps_2$-good with respect to $E_{\text{odd}}(A_3, B_3)$. Since $k/2$ is odd, we can apply the second assertion in Corollary~\ref{cor:good} and find a perfect matching $M_3$ in $H[A_3\cup B_3]$ (here we do not require $|A_3|=|B_3|\equiv 0 \pmod{k}$). Thus, $M_1 \cup M_2 \cup M_3$ is our desired perfect matching in $H$.
(Note that this covers the case when
$\dB (A_1,B_1) \not \in \mathcal H_{\text{ext}} (n,k)$, since $e_0 = \emptyset$ in this case.)

If $e_0\cap V_0 \neq \emptyset$, we need to apply \textbf{Step~5}. As in Case~1a, we remove a matching $M_3$ of size at most $2k$ containing all the vertices of $e_0\cap V_0$ such that $A_4:= A_3 \setminus V(M_3)$ and $ B_4:= B_3 \setminus V(M_3)$ have the same size. Finally in \textbf{Step~6} we find a perfect matching $M_4$ of $H[A_4\cup B_4]$ by the second assertion in Corollary~\ref{cor:good}.
Thus, $M_1 \cup M_2 \cup M_3 \cup M_4$ is a perfect matching in $H$.

\textbf{Case 2b: $d \equiv 2 \pmod{4}$.}  We remove $e_0$ immediately. Let $A'_2:= A_2\setminus e_0$ and $B'_2: = B_2 \setminus e_0$. Since $k\equiv 2 \pmod{4}$ and $r_0$ is even, we have $k- 2 r_0 \equiv 2 \pmod{4}$.  Consequently $|A'_2| - |B'_2|
= (|A_2|-r_0)-(|B_2|-k+r_0)= d + (k- 2r_0) \equiv 0 \pmod{4}$.
We then follow the procedure of Case~2a (since $e_0$ has been removed, we can skip Step~5).
\endproof

\subsection{$k$ is odd}
Let $H$ be a $k$-uniform hypergraph such that it $\eps$-contains $\dB$ or $\doB$.

Recall that $\doB$ is the $n$-vertex $k$-uniform hypergraph on $V= A\cup B$ such that $|A| = \lfloor n/2 \rfloor$, $|B| = \lceil n/2 \rceil$, with edge set $E_{\text{even}}(A, B)$. Since $k$ is odd, $\dB$ can be viewed as the $n$-vertex $k$-uniform hypergraph
on $V= A\cup B$ such that $|A| = \lceil n/2 \rceil$, $|B| = \lfloor n/2 \rfloor$, with edge set $ E_{\text{even}}(A, B)$.
We thus assume that $V(H)= A\cup B$ such that either $|A|= \lfloor n/2 \rfloor$ or $|A| = \lceil n/2 \rceil$ and $|E_{\text{even}}(A, B)\setminus E(H)|\le \eps n^k$.

Our \textbf{Step~1} is the same as in Section~\ref{sec:ext1}. After applying Step 1 we have a partition $A_1,B_1$ of $V(H)$.
If $\doB (A_1,B_1) \not \in \mathcal H_{\text{ext}} (n,k)$ then, by definition of $\mathcal H_{\text{ext}} (n,k)$, $|A_1|$ is even.
 Thus, in this case $|A_1| \bmod k-1$ is even.

If $|A_1| \bmod k-1$ is odd, then we need \textbf{Step~2}: find an $(A_1, B_1)$-odd edge $e_0$. Note that in this case $\doB (A_1,B_1) \in \mathcal H_{\text{ext}} (n,k)$, and thus our minimum $\ell$-degree condition ensures we can find such an edge $e_0$.

Our \textbf{Step~3} is again the same as in Section~\ref{sec:ext1}.
(Note though, if $|A_1| \bmod k-1$ is odd, then we introduced $e_0$. Thus in this case we select $M_1$ to cover $V_0 \backslash e_0$
so that $M_1$ is disjoint from $e_0$.)
%Next, remove a matching $M_1$ of size $|M_1|\le \eps_1 n$ containing all the vertices of $V_0\setminus e_0$.  As usual, define $A_2 := A_1\setminus V(M_1)$ and $B_2 := B_1\setminus V(M_1)$.
Since each edge in the matching $M_1$ is an $A_1^r B_1^{k-r}$ edge for some even $r\le k-1$, it follows that $|A_1| \bmod k-1$ and $|A_2| \bmod k-1$ have the same parity.

Assume that $|A_2|\equiv s \pmod{k-1}$. In \textbf{Step~4}, if $s$ is even, then we simply remove an arbitrary $A_2^s B_2^{k-s}$ edge $e_2$ and let $M_2=\{ e_2\}$. If $s$ is odd, then we remove $e_0$, which is an $A_2^{r_0} B_2^{k- r_0}$ edge for some odd $r_0$.
Set $A'_2 :=A_2 \backslash e_0$. Thus, $|A'_2| \equiv s-r_0 \bmod k-1$ and since $s$, $r_0$ are odd, $s':=|A'_2| \bmod k-1$ is even.
Select an arbitrary $A_2^{s'} B_2^{k-s'}$ edge $e_2$ that is disjoint from $e_0$ and set $M_2=\{e_0, e_2\}$.

Let $A_3 := A_2\setminus V(M_2)$ and $B_3 := B_2\setminus V(M_2)$. The choice of $M_2$ is such that
 $|A_3|\equiv 0 \pmod{k-1}$. We skip Step~5 and proceed to \textbf{Step~6}. Arbitrarily partition $B_3$ into $B^1_3$ and $B^2_3$ such that $|B^1_3| = |A_3|/(k-1)$ (this is possible because $|A_3|\approx |B_3|\approx n/2$).
 Note that $|A_3|+|B_3|\equiv 0 \bmod k$. Hence, as $|A_3|+|B^1 _3|=k|A_3|/(k-1) \equiv 0 \bmod k$, we have that $|B^2_3|\equiv 0 \bmod{k}$. Let $H_1 := H[A_3\cup B^1_3]$ and $H_2 := H[B^2_3]$.
Since $|A_3| + |B_3^1|\ge (1/2 - 2k\eps_1) n k/(k-1)\ge n/2$, by Claim~\ref{clm:HU}, all the vertices of $H_1$ are $(2^{k-1} \eps_2)$-good with respect to $K_{k-1}(A_3, B^1_3)$. Since $k\ge 3$ (because $k\ge 2$ is odd),  we have $|B_3^2| \approx \frac{n}{2} \frac{k-2}{k-1}\ge \frac{n}{2k}$. By Claim~\ref{clm:HU},
 all the vertices of $H_2$ are $((2k)^{k-1} \eps_2)$-good with respect to $K^k[B^2_3]$. We therefore apply Lemma~\ref{goodlem:r} to $H_1$ (with $r=k-1$) and to $H_2$ (with $r=k$) to obtain a perfect matching $M_3$ of $H_1$ and a perfect matching $M'_3$ of $H_2$. Thus $M_1\cup M_2\cup M_3\cup M'_3$ is a perfect matching of $H$. \qed

%%%%%%%%%%%%%%%%%%
\section{The non-extremal case}\label{sectionnon}

In this section we prove Theorem~\ref{nonexthm1}.
%which says that hypergraphs $H$ which satisfy the minimum $\ell$-degree condition in Theorem~\ref{mainthm} are either `close' to one of our extremal examples or contain a reasonably small absorbing set that can absorb any small set of vertices from $V(H)$.
Let $\alpha >0$ and $r, \ell \in \mathbb N$ such that $2r \leq \ell \leq 4r-1$. Given a $4r$-uniform hypergraph $H$ on $n$ vertices such that $\delta _{\ell} (H) \geq \left( \frac{1}{2}-\a \right) \binom{n-\ell}{4r-\ell}$, by Proposition~\ref{prop:deg}, we have $\delta _{2r} (H) \geq \left( \frac{1}{2}- \a \right) \binom{n-2r}{2r}$. Thus, in order to prove Theorem~\ref{nonexthm1} it suffices to prove the following result.
\begin{thm}\label{thm:extabs}
Given any $\eps >0$ and $r \in \mathbb N$, there exist $\a, \xi >0$ and $n_0 \in \mathbb N$ such that the following holds.
Suppose that $H$ is a $4r$-uniform hypergraph on $n \geq n_0$ vertices where $4r$ divides $n$. If
$$\delta _{2r} (H) \geq \left( \frac{1}{2}-\a \right) \binom{n - 2r}{2r}$$
then $H$ is $\eps$-close to $\mathcal B_{n,4r}$ or $\overline{\mathcal B}_{n,4r}$, or $H$ contains
a matching $M$ of size $|M| \le \xi n/(4r)$ that absorbs any set $W\subseteq V(H) \setminus V(M)$ such that $|W| \in 4r\mathbb{N}$ with $|W| \le \xi^2 n$.
\end{thm}
Theorem~\ref{thm:extabs} immediately follows from Lemmas~\ref{lem:abs}--\ref{lem:GH}.
Following the ideas in \cite{rrs2,rrs}, we first show in Lemma~\ref{lem:abs} that in order to find the absorbing set described in Theorem~\ref{thm:extabs}, it suffices to prove that there are at least $\xi n^{8r}$ absorbing $8r$-sets for every fixed $4r$-set from $V(H)$.

\begin{lemma}[Absorbing Lemma]
\label{lem:abs}
Given $0<\xi \ll 1$ and an integer $k \geq 2$, there exists an $n_0 \in \mathbb N$ such that the following holds. Consider a $k$-uniform
hypergraph $H$ on $n \geq n_0$ vertices. Suppose that any $k$-set of vertices $Q \subseteq V(H)$ can be absorbed by at least
$\xi n^{2k}$ $2k$-sets of vertices from $V(H)$. Then $H$ contains a matching $M$ of size $|M|\leq \xi n/k$ that absorbs any set
$W \subseteq V(H) \backslash V(M)$ such that $|W|\in k \mathbb N$ and $|W|\leq \xi ^2 n$.
\end{lemma}

Given a $2r$-uniform hypergraph $H$ (for some $r\geq2$), we define the graph $G(H)$ with vertex set $\binom{V(H)}{r}$ in which two vertices $x_1\dots x_r, y_1\dots y_r \in V(G(H))$ are adjacent if and only if $x_1\dots x_r y_1 \dots y_r \in E(H)$.
When it is clear from the context, we will often refer to $G(H)$ as $G$.
\begin{lemma}[Lemma on $G$]
\label{lem:G}
Given any $\b>0$ and an integer $r\geq 2$, there exist $\a , \xi >0$, and $n_0 \in \mathbb N$ such that the following holds.
Suppose that $H$ is a $2r$-uniform hypergraph on $n \geq n_0$ vertices so that $2r$ divides $n$ and
$$\d_{r}(H) \ge \left(\frac12 - \a \right)\binom{n-r}{r}.$$
Set $G:= G(H)$ and $N:= \binom{n}{r}$ (then $N$ is even because $2r$ divides $n$). Then at least one of the following assertions holds.
\begin{itemize}
\item $G = K_{\frac{N}{2}, \frac{N}{2}} \pm \b N^2$ or $\overline{G} = K_{\frac{N}{2}, \frac{N}{2}} \pm \b N^2$; in other words, either $G$ or $\overline{G}$ becomes a copy of $K_{\frac{N}{2}, \frac{N}{2}}$ after adding or deleting at most $ \beta N^2$ edges.
\item There are at least $\xi n^{4r}$ absorbing $4r$-sets in $\binom{V(H)}{4r}$ for every $2r$-subset of $V(H)$.
\end{itemize}
\end{lemma}

\begin{lemma}
\label{lem:GH}
Given any $\eps >0$ and $r \in \mathbb N$, there exist $\b>0$ and $n_0 \in \mathbb N$ such that the following holds.
Suppose that $H$ is a $4r$-uniform hypergraph on $n \geq n_0$ vertices where $4r$ divides $n$.
Suppose further that $G:= G(H)$ satisfies $G = K_{\frac{N}{2}, \frac{N}{2}} \pm \b N^2 $ or $\overline{G} = K_{\frac{N}{2}, \frac{N}{2}} \pm \b N^2 $, where $N:= \binom{n}{2r}$. Then $H$ is $\eps$-close to $\mathcal B_{n,4r}$ or $\overline{\mathcal B}_{n,4r}$.
\end{lemma}
Notice we have stated Lemmas~\ref{lem:abs} and~\ref{lem:G} in a more general setting than we require.
(That is, we consider $k$-uniform hypergraphs in Lemma~\ref{lem:abs} for all $k \geq 2$ and
$2r$-uniform hypergraphs  in Lemma~\ref{lem:G} for $r \geq 2$.)
 However, for Lemma~\ref{lem:GH}, our proof is such that we can only consider
$4r$-uniform hypergraphs for $r \in \mathbb N$. (This is the main obstacle in extending our proof to work for all
$2r$-uniform hypergraphs.)
The rest of the section is devoted to the proof of Lemmas~\ref{lem:abs}--\ref{lem:GH}.

%%%%%%%%%%%%%%%%%%%%%%%%%
\subsection{Proof of Lemma~\ref{lem:abs}}
For a $k$-set $Q \subseteq V(H)$, let $L_Q$ denote the family of all absorbing $2k$-sets for $Q$. By assumption, $|L_Q| \ge \xi n^{2k}$.
Let $F$ be the family of $2k$-sets obtained by selecting each of the $\binom{n}{2k}$ elements of $\binom{V(H)}{2k}$
independently with probability $p:= \xi / n^{2k-1}$.
Then
$$\mathbb{E}(|F|) = p \binom{n}{2k}  < \frac{\xi}{(2k)!} n \ \text{ and } \
\mathbb{E}(|L_Q\cap F|)\ge p \,\xi n^{2k}= \xi^2 n$$ for every set $Q\subseteq \binom{V(H)}{k}$.

Since $n$ is sufficiently large, Proposition~\ref{chernoff} implies that with high probability we have
\begin{equation}\label{eq:F}
    |F|\le 2 \mathbb{E}(|F|)< \frac{2\xi}{(2k)!} n,
\end{equation}
\begin{equation}\label{eq:LF}
   |L_Q\cap F| \ge \frac12 \, \mathbb{E}(|L_Q\cap F|) \ge \frac{\xi^2}{2} n \quad \text{for all  } Q\in \binom{V(H)}{k}.
\end{equation}
Let $Y$ be the number of intersecting pairs of members of $F$. Then
\[
\mathbb{E}(Y)\le p^2 \binom{n}{2k} 2k \binom{n}{2k-1}\le \frac{\xi^2 n}{(2k-1)!(2k-1)!}.
\]
By Markov's bound, the probability that $Y\le \frac{2 \xi^2}{(2k-1)!(2k-1)!} n$ is at least $\frac{1}{2}$. Therefore we can find a family $F$ of $2k$-sets satisfying \eqref{eq:F} and \eqref{eq:LF} and having at most $\frac{2 \xi^2}{(2k-1)!(2k-1)!} n$
intersecting pairs. Removing all non-absorbing $2k$-sets and one set from each of the intersecting pairs in $F$, we obtain a family $F'$ of disjoint absorbing $2k$-sets such that $|F'|\le |F|\le \frac{2\xi}{(2k)!} n \leq \xi n/2k$ and for all $Q\in \binom{V(H)}{k}$,
\begin{equation}\label{eq:LF'}
    |L_Q\cap F'| \ge \frac{\xi^2}{2} n - \frac{2 \xi^2}{(2k-1)!(2k-1)!} n > \frac{\xi^2}{k} n.
\end{equation}
Since $F'$ consists of disjoint absorbing sets and each absorbing set is covered by a matching, $V(F')$ is covered by a matching $M$.
Now let $W\subseteq V(H)\backslash V(F')$ be a set of at most $\xi^2 n$ vertices such  that $|W|= k\ell$ for some $\ell \in \mathbb N$. We arbitrarily partition $W$ into $k$-sets $Q_1, \dots, Q_{\ell}$. Because of \eqref{eq:LF'}, we are able to absorb each $Q_i$ with a different $2k$-set from $L_{Q_i}\cap F'$. Therefore $V(F')\cup W$ is covered by a matching, as desired.

%%%%%%%%%%%%%%%%%%%%%
\subsection{Proof of Lemma~\ref{lem:G}}
Given $\b> 0$, we choose additional constants $\g, \a, \xi$ such that
\begin{equation}
\label{eq:agxi}
0< \xi \ll \a \ll \g \ll \beta.
\end{equation}
Without loss of generality we may assume that $\b \ll 1/r$.
We also assume that $n$ is sufficiently large.

Let $Q\subseteq V(H)$ be a $2r$-set.
It is easy to see that if $Q$ has at least $\g^3 n^{2r}$ absorbing $2r$-sets then $Q$ has at least $\xi n^{4r}$ absorbing $4r$-sets. Indeed, let $P$ be an absorbing $2r$-set for $Q$. Then $P\cup e$ is an absorbing $4r$-set for $Q$ for any edge $e\in E(H - (P\cup Q))$. Since $n$
is sufficiently large,
$$|E(H)| \ge \left(\frac12 - \a \right)\binom{n-r}{r}\times \frac{\binom{n}{r}}{\binom{2r}{r}} = \left( \frac{1}{2}- \a \right) \binom{n}{2r}.$$
Hence, as $n$ is sufficiently large, there are at least
\[
\left(\frac12 - \a \right)\binom{n}{2r} - 4r\binom{n}{2r-1} \ge \frac{n^{2r}}{4(2r)!}
\]
edges in $H- (P\cup Q)$. Since an absorbing $4r$-set may be counted at most $\binom{4r}{2r}$ times when counting the number of $P, e$, there are at least
\[
\g^3 n^{2r} \times  \frac{n^{2r}}{4(2r)!} \times\frac{1}{\binom{4r}{2r}} \stackrel{(\ref{eq:agxi})}{\ge}  \xi n^{4r}
\]
absorbing $4r$-sets for $Q$.

Therefore, in order to prove Lemma~\ref{lem:G}, it suffices to prove the following two claims.

\begin{claim}\label{clm:ab}
If either of the following cases holds, then we can find $\g^3 n^{2r}$ absorbing $2r$-sets or $\g^3 n^{4r}$ absorbing $4r$-sets for every
$2r$-set $Q\in \binom{V(H)}{2r}$.
\begin{description}
\item[Case (a) ] For any $r$-tuple $\underline{a} \in \binom{V(H)}{r}$, there are at least $(\frac{1}{2} + \g) \binom{n}{r}$
$r$-tuples $\underline{b} \in \binom{V(H)}{r}$ such that $| N_H( \underline{a} ) \cap N _H (\underline{b}) | \ge \g \binom{n}{r}$.
\item[Case (b)] $| \{ \underline{a} \in \binom{V(H)}{r}: d_H(\underline{a})\ge (\frac12 + \g)\binom{n}{r} \} |\ge 2\g \binom{n}r$.
\end{description}
\end{claim}

\begin{claim}\label{clm:abG}
If neither Case (a) or Case (b) holds, then $G = K_{\frac{N}{2}, \frac{N}{2}} \pm \b N^2$ or
$\overline{G} = K_{\frac{N}{2}, \frac{N}{2}} \pm \b N^2$.
\end{claim}

\noindent
{\bf Proof of Claim~\ref{clm:ab}.}
Given a $2r$-set $Q= \{x_1, \dots , x_r , y_1 , \dots , y_r \} \subseteq V(H)$, we will consider two types of absorbing sets for $Q$:
\begin{description}
\item[Absorbing $2r$-sets] These consist of a single edge $x'_1\dots x'_r y'_1 \dots y'_r \in E(H)$ with the property that
both $x_1 \dots x_r x'_1 \dots x'_r$ and $y_1 \dots y_r y'_1 \dots y'_r$ are edges of $H$.

\item[Absorbing $4r$-sets] These consist of distinct vertices $x'_1, \dots, x'_r$, $y'_1 , \dots , y'_r$, $w'_1, \dots , w'_r $, $z'_1 \dots , z'_r \in V(H)$ such that $x'_1 \dots x'_r w'_1 \dots w'_r\,$, $y'_1 \dots y'_r z'_1 \dots  z'_r$ and $ w'_1 \dots w'_r z'_1 \dots  z'_r$ are edges in $H$. Furthermore, $x_1 \dots x_r x'_1 \dots x'_r$ and $y_1 \dots y_r y'_1 \dots y'_r$ are also edges of $H$ (see Figure~1).
\end{description}
\begin{figure}\label{picture}
\begin{center}\footnotesize
\includegraphics[width=1\columnwidth]{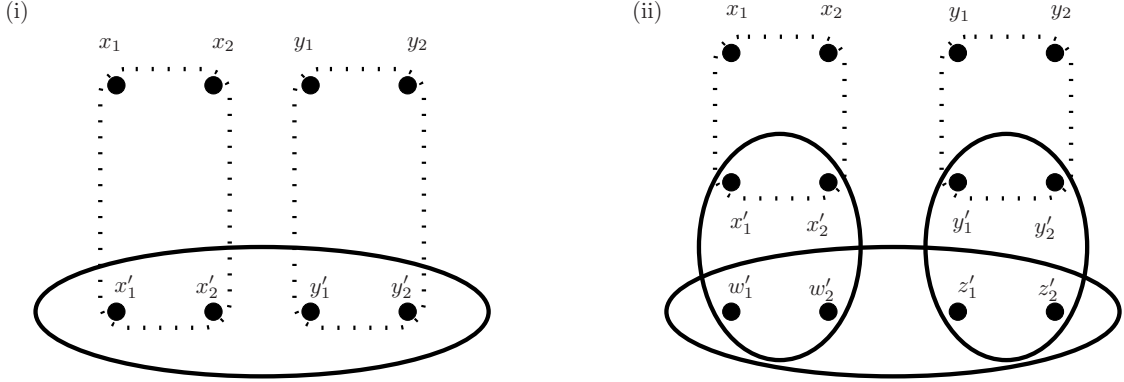}
\caption{The (i) absorbing $2r$-set and (ii) absorbing $4r$-set in the case when $r=2$.}
\end{center}
\end{figure}

Write $\underline{x}:=x_1 \dots x_r$ and $\underline{y}:=y_1 \dots y_r$. For any two (not necessarily disjoint) $r$-tuples $\underline{a}$, $\underline{b} \in \binom{V(H)}{r}$ we call $\underline{a}$
a \emph{good $r$-tuple for $\underline{b}$} if $|N_H (\underline{a}) \cap N_H (\underline{b})|\geq \g \binom{n}{r}/2$.
 We first observe that $Q$ has at least $ \g^3 n^{2r}$ absorbing $2r$-sets if there are
\begin{align}\label{eq:good}
 & \text{at least $\frac{\g}{2} \binom{n}{r}$ good $r$-tuples in $N_H(\underline{x})$ for $\underline{y}$}\\
 & \text{or  at least $\frac{\g}{2} \binom{n}{r}$ good $r$-tuples in $N_H(\underline{y})$ for $\underline{x}$.}\nonumber
\end{align}
Indeed, assume that there are at least $\g \binom{n}{r}/2$ good $r$-tuples in $N_H(\underline{x})$ for $\underline{y}$.
There are at most $r\binom{n}{r-1}$ $r$-tuples in $\binom{V(H)}{r}$ that contain at least one  element from $\{y_1, \dots , y_r \}$.
So there are at least $\g \binom{n}{r}/2 - r\binom{n}{r-1}$ $r$-tuples in $N_H(\underline{x})$ that are good for $\underline{y}$ and disjoint from $\underline{y}$. Let us pick such an $r$-tuple $\underline{x}'=(x'_1 \dots x'_r)$. Thus,
$|N_H(\underline{x}')\cap N_H(\underline{y})|\ge \g \binom{n}{r}/2$. We pick $\underline{y}'=(y'_1\dots y'_r) \in
N_H(\underline{x}')\cap N_H(\underline{y})$ such that $\underline{y}'$ is disjoint from $\underline{x}$.
Note that there are at least $\g \binom{n}{r}/2 - r\binom{n}{r-1}$ choices for $\underline{y}'$.
Notice that the $2r$-set $\{x'_1, \dots , x'_r, y'_1 , \dots , y'_r \}$ is an absorbing set for $Q$ since
$x'_1\dots x'_r y'_1 \dots y'_r \,$,  $x_1 \dots x_r x'_1 \dots x'_r$ and $y_1 \dots y_r y'_1 \dots y'_r$ are edges in $H$.
Since an absorbing $2r$-set may be counted $\binom{2r}{r}$ times, this argument implies that there are at least
\[
\left( \frac{\g}{2} \binom{n}{r} - r\binom{n}{r-1} \right)^2 \frac{1}{\binom{2r}{r}} \ge \g^3  n^{2r}
\]
absorbing $2r$-sets for $Q$. We reach the same conclusion when there are at least $\gamma \binom{n}{r}/2$ good $r$-tuples in
$N_H (\underline{y})$ for $\underline{x}$.

\medskip
Now assume that Case (a) holds. This implies that there are at least $(\frac{1}{2} + \g) \binom{n}{r}$  good $r$-tuples for
$\underline{x}$.
 By the minimum $r$-degree condition, $d_H (\underline{y})\ge (\frac12 - \a)\binom{n}{r}$. So there are at least
 $(\g - \a)\binom{n}{r}\ge \g \binom{n}{r}/2$ $r$-tuples in $N_H(\underline{y})$ that are good for $\underline{x}$.  Thus \eqref{eq:good} holds and consequently $Q$ has at least $\g^3 n^{2r}$ absorbing $2r$-sets.

\medskip
Next assume Case (b) holds. Let $\Lambda:=\{ \underline{a} \in \binom{V(H)}{r}: d_H(\underline{a})\ge (\frac12 + \g)\binom{n}{r} \}$.
So by assumption, $|\Lambda |\ge 2\g \binom{n}{r}$. We also assume that \eqref{eq:good} fails (otherwise we are done). Every
$r$-tuple $\underline{a} \in \Lambda$ is good for arbitrary $\underline{b} \in \binom{V(H)}{r}$ because
$| N_H(\underline{a})\cap N_H(\underline{b}) |\ge (\g - \a )\binom{n}{r}\ge \g \binom{n}{r} /2$.
Hence $|\Lambda \cap N_H(\underline{y})|< \g \binom{n}{r}/2$.
On the other hand, less than $\g \binom{n}{r}/2$ $r$-tuples in $N_H(\underline{x})$ are good for $\underline{y}$ and consequently at least $(\frac12 - \a)\binom{n}{r} - \frac{\g}{2} \binom{n}{r} $ $r$-tuples $\underline{x}' \in N_H(\underline{x})$
satisfy $|N_H(\underline{x}')\cap N_H(\underline{y})|< \g \binom{n}{r}/2$.
We pick such an $r$-tuple $\underline{x}'$ that is disjoint from $\underline{y}$; there are at least $(\frac12 - \a)\binom{n}{r} - \frac{\g}{2} \binom{n}{r} - r\binom{n}{r-1}\ge (\frac12 - \g)\binom{n}r$ $r$-tuples with this property. Since
\[
| N_H(\underline{x}') \cup N_H(\underline{y}) | \ge 2 \left(\frac12 - \a \right)\binom{n}{r} - \frac{\g}{2} \binom{n}{r}
 \ge \binom{n}{r} - {\g} \binom{n}{r},
\]
it follows that
\begin{align}
| \Lambda \cap N_H(\underline{x}')| & \ge |\Lambda| - | \Lambda \cap N_H(\underline{y})| - | \overline{(N_H(\underline{x}')
\cup N_H(\underline{y}))} | \nonumber \\
& \ge  2\g \binom{n}{r} - \frac{\g}{2} \binom{n}{r} - {\g} \binom{n}{r} = \frac{\g}2 \binom{n}{r}. \label{eq;x'y'}
\end{align}
Now pick any $\underline{w}' \in \Lambda \cap N_H(\underline{x}')$ that is disjoint from $Q$.
(Note there are at least $\frac{\g}2 \binom{n}{r} - 2r\binom{n}{r-1} \ge \frac{\g}{3} \binom{n}{r}$ choices for $\underline{w}'$.)
Next pick an $r$-tuple $\underline{y}'\in N_H(\underline{y})$ such that $\underline{y}'$ is disjoint from
$\underline{x}$, $\underline{x}'$ and $\underline{w}'$. (There are at least $(\frac12 - \a) \binom{n}{r} - 6r\binom{n}{r-1}
\ge (\frac12 - \g) \binom{n}{r}$ choices for $\underline{y}'$ here.)
By the definition of $\Lambda$,  there are at least $(\g - \a) \binom{n}{r}$ pairs in $N_H(\underline{w}'
)\cap N_H(\underline{y}')$.
We pick $\underline{z}'\in N_H(\underline{w}')\cap N_H(\underline{y}')$  such that $\underline{z}'$ is disjoint from
$\underline{x}$, $\underline{y}$ and $\underline{x}'$. (There are at least $(\g - \a) \binom{n}{r} - 6r\binom{n}{r-1}
\ge \g  \binom{n}{r}/2$ choices for $\underline{z}'$ here.)

Let $S$ denote the $4r$-set consisting of the vertices contained in $\underline{x}'$, $\underline{y}'$, $\underline{w}'$
 and $\underline{z}'$. By the choice of $\underline{x}'$, $\underline{y}'$, $\underline{w}'$
 and $\underline{z}'$, $S$ is an absorbing $4r$-set for $Q$.

In summary, there are at least $(\frac12 - \g) \binom{n}{r} $ choices for $\underline{x}'$, at least
$\frac{\g}{3} \binom{n}{r} $ choices for $\underline{w}'$, at least $(\frac12 - \g) \binom{n}{r}$ choices for
$\underline{y}'$  and at least $\frac{\g}{2} \binom{n}{r}$ choices for $\underline{z}'$. Since each absorbing $4r$-set may be counted $\binom{4r}{r} \binom{3r}{r} \binom{2r}{r}$ times, there are at least
\[
\left[ \left(\frac12 - \g \right) \binom{n}{r}  \right]^2 \frac{\g}3 \binom{n}{r} \frac{\g}{2} \binom{n}{r}
\times \frac{1}{\binom{4r}{r} \binom{3r}{r} \binom{2r}{r}}
 \stackrel{(\ref{eq:agxi})}{\geq } \g ^3 n^{4r}
\]
absorbing $4r$-sets for $Q$, as desired.
\endproof

Claim~\ref{clm:abG} follows from the following lemma (by letting $G= G(H)$) immediately.

\begin{lemma}\label{lem:abG}
For any $\beta > 0$, there exist $\gamma > 0$ and $n_0\in \mathbb{N}$
such that following holds. Let $G=(V,E)$ be a graph on an even $N\ge n_0$ number of vertices such that
$\delta(G)\ge (1/2 - \gamma) N$. In addition, $G$ satisfies

\begin{description}
\item[(a)] There exists $ {a} \in V$ such that at most $(\frac{1}{2} + \gamma) N$
vertices ${b} \in V$ satisfy $| N(a) \cap N(b) | \ge \gamma N$.

\item[(b)] $| \{ v \in V: d(v)\ge (\frac12 + \gamma)N \} |< 2\gamma N$.
\end{description}

Then either $G= K_{N/2, N/2} \pm \beta N^2$ or $\overline{G} = K_{N/2, N/2} \pm \beta N^2$.
\end{lemma}

\begin{proof}
Let $A:= N(a)$ and $B:=\{ {b} \in V: | A \cap N (b) | < \g N \}$.  Then $|A| \ge (\frac12 - \gamma)N$ and $|B|\ge (\frac12 - \g) N$.

We also need an upper bound on $|A|$. Fix $b \in B$. Since $|N(b)| \ge (\frac12 - \gamma) N$,  we have
\[
|A| + \left(\tfrac12 - \gamma \right) N \le |A| + |N(b)| = |A\cup N(b)| + |A\cap N(b)| \le N + \g N,
\]
which gives $|A|\le (\tfrac12 + 2 \g) N$.

Let $e(A, B)$ denote the number of \emph{ordered} pairs $a, b$ such that $a\in A$, $b\in B$, and $ab\in E$.
(Therefore, if $a,b \in A\cap B$ and $ab \in E$ then $ab$ counts twice to the value of $e(A,B)$.)
By the definition of $B$, we have $e(A, B)\le \g N |B| $. Since $\d(G)\ge (\frac12 - \g) N$ we have that
\begin{equation}\label{eq:ABbar}
e (\bar{A}, B)\ge (1/2 - 2\g) N  |B|.
\end{equation}
where, as usual $\bar{A}:= V\setminus A$. Next we show that $e(\bar{A}, \bar{B})$ is very small.

\begin{claim}\label{clm:AcB}
$e(\bar{A}, \bar{B})\le 8 \sqrt{\g} |\bar{A}| |\bar{B}|$.
\end{claim}
\begin{proof}
Assume for a contradiction that the claim is false.
Set $A_1 := \{{x} \in \bar{A}:  d({x}, \bar{B}) \ge 4\sqrt{\g} |\bar{B}| \}$. By assumption
\[
8\sqrt{\g} |\bar{A}| |\bar{B}| \le e(\bar{A} , \bar{B} ) \le |A_1| |\bar{B}| +4\sqrt{\g}|\bar{B}|  |\bar{A}|,
\]
which gives that $|A_1|\ge 4\sqrt{\g} |\bar{A}|$. By \eqref{eq:ABbar}, as $|\bar{A}|\le (\frac12 + \g) N $, we derive that
\begin{equation}\label{eq:AcB}
e(\bar{A}, B)\ge (\tfrac12 - 2 \g) N |B| \ge (1-6\g)(\tfrac12+ \g) N |B| \ge (1 - 6\g) |\bar{A}| |B|.
\end{equation}
Let $A_2 := \{ {x} \in \bar{A}:  d({x}, {B}) \ge (1- 3\sqrt{\g}) |B| \}$. We claim that $|A_2|\ge (1 - 3\sqrt{\g}) |\bar{A}| $. Indeed, for convenience, consider $\bar{e}(\bar{A}, B)$, the number of ordered pairs $a, b$ such that $a\in \bar{A}$, $b\in B$, and $ab\not\in E$.
If $|A_2|< (1 - 3\sqrt{\g}) |\bar{A}| $, then $\bar{e}(\bar{A}, B)\ge 3\sqrt{\g} |\bar{A}| 3\sqrt{\g} |B| = 9\g |\bar{A}| |B|$, contradicting \eqref{eq:AcB}.

Let $A_0 := A_1 \cap A_2$. We have $|A_0|\ge (4\sqrt{\g} - 3\sqrt{\g}) |\bar{A}|$. Since $|\bar{A}|\ge N /3$ and $\g \le 1/36$,  we derive that $|A_0| \ge \sqrt{\g}N/3\ge 2\g N$. For every ${x}\in A_0$,
we have
\begin{align*}
d({x}) & =  d(x, B) + d(x, \bar{B}) \\
& \ge (1- 3\sqrt{\g}) |B| + 4\sqrt{\g} |\bar{B}| = (1- 7\sqrt{\g}) |B| + 4\sqrt{\g} N \\
& \ge \left(\frac12 - \frac{7}{2}\sqrt{\g} + 4\sqrt{\g} - \g \right)N
\ge  \left(\frac12 + \g \right)N.
\end{align*}
(The penultimate inequality follows since $ |B|\ge  (\tfrac12 - \g) N$.)
This is a contradiction to the assumption \textbf{(b)}.
\end{proof}

Now we go back to the proof of Lemma~\ref{lem:abG}. We separate the cases by whether $|\overline{A\cup B}| \le {\g}^{1/4} N$ or not.

First assume that $|\overline{A\cup B}| \le {\g}^{1/4}N$.  Since $(\frac12 - \g)N \le |A|\le (\frac12 +2\g) N$, we can find a set $V_1\subseteq V(G)$ of size $N/2$ such that $|V_1 \triangle A| \le 2\g N$. Let $V_2 := V(G) \backslash V_1$. Thus,
\begin{align*}
e (V_1, V_2) & \le e(A\cap V_1, B\cap V_2) + e(V_1\setminus A, V_2) + e(V_1, V_2\cap A) + e(V_1, V_2\setminus (A\cup B))  \\
& \le e(A, B) + \left( |V_1\setminus A| \frac{N}2 + |V_2\cap A|\frac{N}2 \right) + |\overline{A\cup B}| \frac{N}2\\
& \le  \g N |B|+ 2\g N \frac{N}2 + \g^{1/4} N \frac{N}{2} \le \g^{1/4} N^2.
\end{align*}
Since $\d(G)\ge (\frac12 - \g)N $, we derive that for $i=1, 2$,
\[
2e(V_i) = e(V_i, V_i)\ge \tfrac{N}2 (\tfrac12- \g) N - \g^{1/4} N^2 \geq  \tfrac{N^2}{4} - 2 \g^{1/4} N^2.
\]
Thus, we can delete at most $\g^{1/4} N^2$ edges between $V_1$ and $V_2$ and add at most $\g^{1/4} N^2$ edges in each of $V_1$ and $V_2$ to turn $G$ into a graph consisting of two vertex-disjoint cliques; one on $V_1$, the other on $V_2$. In other words, $\overline{G} = K_{N/2, N/2} \pm 3\g^{1/4} N^2$.

\medskip

Now assume that $|\overline{A\cup B} |\ge {\g}^{1/4} N$. Then $| B\setminus A| \le |\bar{A}| - {\g}^{1/4} N \le (\frac12 + \g - {\g}^{1/4}) N $. By Claim~\ref{clm:AcB}, $e(B\setminus A, \overline{A\cup B})\le e(\bar{A}, \bar{B})\le 8\sqrt{\g} |\bar{A}| |\bar{B}| \le 8\sqrt{\g} N^2$. Together with $e (B\setminus A, A)\le e (B, A)\le |B| \g N\le \g N^2$, it gives that
\begin{align*}
|B\setminus A| (\tfrac12 - \g) N & \le e(B\setminus A, V)\le e(B\setminus A, B\setminus A) + e (B\setminus A, A) + e(B\setminus A, \overline{A\cup B}) \\
& \le |B\setminus A| (\tfrac12 + \g - {\g}^{1/4}) N + {\g} N^2  + 8 \sqrt{\g} N^2.
\end{align*}
This implies that $({\g}^{1/4} - 2\g) N |B\setminus A|\le 9\sqrt{\g} N^2$ and so $|B\setminus A|\le 10{\g}^{1/4} N$. Similarly we can show that $|A\setminus B|\le 10{\g}^{1/4} N$. Now pick a set $V_1\subseteq V(G)$ of size $N/2$ such that $|V_1 \cap A|$ is maximized.  Thus, $|V_1 \setminus A| \le \g N$. Then, as $e(A\cap B, A)\le e(B, A)\le \g N^2$, we have
%$|(V_1\cap A) \setminus (A\cap B)|\le 8 \g^{1/4} N$.
\begin{align*}
e(V_1) &  \le e(V_1\setminus A, V_1) + e(A)\le e(V_1\setminus A, V_1) + e(A\cap B, A) + e(A\setminus B) \\
& \le \g \frac{N^2}{2} +  \g N^2 +  \frac{1}{2} (10 \g^{1/4} N)^2 \le 52 \sqrt{\g} N^2.
\end{align*}

Let $V_2 := V(G) \backslash V_1$. Since $\d(G)\ge (\frac12 - \g)N $, we have
\[
e(V_1, V_2)\ge \frac{N}2 (\tfrac12- \g) N - 52 \sqrt{\g} N^2  \geq  \frac{N^2}{4} - 53 \sqrt{\g} N^2.
\]
Further, by Claim~\ref{clm:AcB} and since $|A\cap V_2| = |A \backslash V_1|\leq 2\g N$,
\begin{align*}
e(V_2)  & \le e(A\cap V_2, V_2) + e(\bar{A}) \le e(A\cap V_2, V_2) + e(\bar{A} \cap \bar{B}, \bar{A}) + e(\bar{A}\cap B)\\
& \le |A\cap V_2| \frac{N}{2} + e(\bar{B}, \bar{A}) + e(B\setminus A) \\
& \le 2\g \frac{N^2}{2} + 8\sqrt{\g} N^2 + \frac{1}{2} (10 \g^{1/4} N)^2 \le 59 \sqrt{\g} N^2.
\end{align*}

Hence, we can add at most $ 53 \sqrt{\g} N^2$ edges between $V_1$ and $V_2$ and delete at most $  52 \sqrt{\g} N^2 +  59 \sqrt{\g} N^2$ edges inside $V_1$ and $V_2$ to turn $G$ into a complete balanced bipartite graph. In other words, $G = K_{N/2, N/2} \pm  164 \sqrt{\g} N^2$.

Since $\gamma \ll \beta$ we conclude that either $G = K_{\frac{N}{2}, \frac{N}{2}} \pm \b N^2$ or $\overline{G} = K_{\frac{N}{2}, \frac{N}{2}} \pm \b N^2$, as desired.
 \end{proof}

This completes the proof of  Lemma~\ref{lem:G}.

\subsection{Proof of Lemma~\ref{lem:GH}}\label{end}\label{seccy}
We need the following structural result and prove it by applying Theorem~\ref{thm:RS} and Theorem~\ref{thm:Si}.

\begin{lemma}[Structure Lemma]\label{lem:str}
For any $\eta >0$ and $r \in \mathbb N$, there exist $\d  >0$ and $n_0 \in \mathbb N$ such that the following holds.
Suppose that $K$ is a complete $2r$-uniform hypergraph on $n \geq n_0$ vertices whose edge set is
partitioned into two sets $R$ (red) and $B$ (blue). Let $\Omega$ denote the collection of all $4r$-subsets $S\subseteq V(K)$ such that there exists a partition of $S= P_1\cup P_2\cup P_3\cup P_4$
where $|P_i|=r$ for all $1\leq i \leq 4$ and
such that exactly one of the four $2r$-sets $P_1 \cup P_2$, $P_2 \cup P_3$, $P_3 \cup P_4$, $P_4 \cup P_1$ is in $R$ or $B$ (the other three are in the other color class). Suppose that
\begin{description}
\item[(i)] $|R|, |B|\ge (\frac12 - \d)\binom{n}{2r}$ and;
\item[(ii)] $|\Omega| \leq \delta n^{4r}$.
\end{description}
Then either $K[R]= \mathcal B_{n,2r} \pm \eta n^{2r}$ or $K[B] =  \mathcal B_{n,2r} \pm \eta n^{2r}$.

\end{lemma}
\begin{proof} %[Proof of Lemma~\ref{lem:str}]
Given $\eta >0$ define additional constants $\delta , \delta _1, \eps$ such that
\begin{align}\label{hiera}
0< \delta \ll \delta _1 \ll \eps \ll \eta , 1/r .
\end{align}

Let $ C^{2r} _4$ denote the expanded $2r$-uniform $4$-cycle. That is, $ C^{2r} _4$ consists of
four disjoint sets $P_1, P_2, P_3, P_4$ of vertices of size $r$, and the edges $P_1 \cup P_2$, $P_2 \cup P_3$, $P_3 \cup P_4$, $P_4 \cup P_1$.
We call a $2$-colored copy of   $ C^{2r} _4$ \emph{bad} if exactly one of its four edges is in $R$ or $B$ (and the other three are in the other color class). A $4r$-set $S \in \binom{V(K)}{4r}$ is \emph{bad} if $K[S]$ contains a bad $ C^{2r} _4$. Thus (ii) says that the number of bad $4r$-sets is at most $\d n^{4r}$.

Observe that if $T_1$ is red copy of $\mathcal C ^{2r} _3$ and $T_2$ is a blue copy of $\mathcal C^{2r} _3$
such that $T_1$ and $T_2$ are vertex-disjoint, then there exists at least one bad copy of $ C^{2r} _4$ whose
vertex set is contained in $V(T_1)\cup V(T_2)$:
Let $\underline{a}, \underline{b}, \underline{c} $ denote the $r$-tuples in $V(T_1)$ such that $\underline{a} \cup \underline{b}$,
$\underline{b}  \cup \underline{c}$, $\underline{c} \cup \underline{a} \in E(T_1)$. Define
$\underline{x}, \underline{y}, \underline{z} \subseteq V(T_2) $ analogously. If there is a $\underline{v} \in \{ \underline{a}, \underline{b}, \underline{c} \}$ such  that $\underline{v} \cup \underline{w}_1 \in R$ and $\underline{v} \cup \underline{w}_2 \in B$
for some $\underline{w}_1, \underline{w}_2 \in \{\underline{x}, \underline{y}, \underline{z} \}$, then we obtained our
desired bad copy of $ C^{2r} _4$. For example, if $\underline{a} \cup \underline{x} \in B$ and
$\underline{a} \cup \underline{z} \in R$, then the edges $\underline{a} \cup \underline{x}, \underline{x} \cup \underline{y},
\underline{y} \cup \underline{z} \in B$ and
$\underline{a} \cup \underline{z} \in R$ induce a bad copy of $ C^{2r} _4$.
Similarly, if there exists $\underline{v} \in \{ \underline{x}, \underline{y}, \underline{z} \}$  such  that $\underline{v} \cup \underline{w}_1 \in R$ and $\underline{v} \cup \underline{w}_2 \in B$
for some $\underline{w}_1, \underline{w}_2 \in \{\underline{a}, \underline{b}, \underline{c} \}$, then we obtain
a bad copy of $ C^{2r} _4$.
If neither of these two cases holds, then all the edges of the form $\underline{v} \cup \underline{w} $ receive the same color, say red
(where $\u{v} \in \{ \u{a}, \u{b} , \u{c} \}$ and $\u{w} \in \{ \u{x}, \u{y} , \u{z} \}$).
But then $\u{a} \cup \u{b} , \u{a} \cup \u{x} , \u{b} \cup \u{y} \in R$ and $\u{x} \cup \u{y} \in B$
induce a bad copy of $ C ^{2r} _4$.

Assume for a contradiction that $K$ contains at least $\delta _1 n^{3r}$ red copies of $\mathcal C^{2r} _3$ and at least
$\delta _1 n^{3r}$ blue copies of $\mathcal C^{2r} _3$. For each red copy $T$ of $\mathcal C^{2r} _3$ in $K$, there are
at most $3r\binom{n}{3r-1}\binom{3r}{2r}\binom{2r}{r}$ blue copies of $\mathcal C^{2r} _3$ in $K$ which contain at least one vertex
from $V(T)$. (The $\binom{3r}{2r} \binom{2r}{r}$ term comes from the fact that, given any $3r$-set $V \subseteq V(K)$, there are
$\binom{3r}{2r} \binom{2r}{r}$ copies of $\mathcal C^{2r} _3$ in $K[V]$.) So there are at least $\delta _1 n^{3r}
-3r\binom{n}{3r-1}\binom{3r}{2r}\binom{2r}{r} \geq \delta _1 n^{3r} /2$ blue $\mathcal C^{2r} _3$ in $K$ that are disjoint from $T$.
Hence, there are at least $\delta _1 ^2 n^{6r} /2$ pairs $T_1, T_2$ of vertex-disjoint copies of $\mathcal C^{2r} _3$ such that
$T_1$ is red and $T_2$ is blue.

Now consider any bad copy $C$ of $C^{2r} _4$. There are $\binom{n-4r}{2r}$ $6r$-subsets of $V(K)$ which contain $V(C)$. For each
such $6r$-set $S$, there are $\binom{6r}{3r}\binom{3r}{2r} ^2 \binom{2r}{r} ^2$ pairs $T',T''$ of vertex-disjoint copies of
$\mathcal C^{2r}_3$ in $K$ such that $V(T') \cup V(T'')=S$. Together, this all implies that the number of bad $4r$-sets is at least
$$\frac{\delta _1 ^2}{2} n^{6r} \times \frac{1}{\binom{n-4r}{2r} \binom{6r}{3r}\binom{3r}{2r} ^2 \binom{2r}{r} ^2} \stackrel{(\ref{hiera})} {>} \delta n^{4r},$$
a contradiction to \text{(ii)} as desired.

Thus, there are less than $\delta _1 n^{3r}$ blue $\mathcal C^{2r}_3$ in $K$ or less than
$\delta _1 n^{3r}$ red $\mathcal C^{2r}_3$ in $K$.
Without loss of generality we assume there are less than $\delta _1 n^{3r}$ blue $\mathcal C^{2r}_3$ in $K$.
So (i) implies that $K[B]$ is an $n$-vertex $2r$-uniform hypergraph with at least $(1/2- \delta) \binom{n}{2r}$ edges and
 less than $\delta _1 n^{3r}$ copies of $\mathcal C^{2r}_3$.
To show $K[B]  = \mathcal B_{n,2r} \pm \eta n^{2r}$, we
will use Theorems~\ref{thm:RS} and~\ref{thm:Si}.

Since $\delta _1 \ll \eps$, Theorem~\ref{thm:RS} implies that we may remove at most $\eps  n^{2r}$ edges from $K[B]$ to obtain a
$\mathcal C^{2r} _3$-free hypergraph $K'[B]$. As $\eps \ll \eta$ and
$$e(K'[B])\geq \left(\frac{1}{2}-\delta \right) \binom{n}{2r}-\eps n^{2r} \geq \left(\frac{1}2-\sqrt{\eps} \right)\binom{n}{2r}$$
we may apply Theorem~\ref{thm:Si} to obtain that $K'[B]=\mathcal B_{n,2r} \pm \eta n^{2r}/2$. Consequently,
$K[B]=\mathcal B_{n,2r} \pm \eta n^{2r}$, as desired.
\end{proof}

\bigskip

Given two disjoint vertex sets $R$ and $B$ we define $K_{R,B}$ to be the complete bipartite graph with vertex classes $R$ and $B$.
 %Usually $A\times B$ denotes the set of all \emph{ordered} pairs $(a, b)$. But  we consider mostly unordered pairs below. Since $A$ and $B$ are disjoint, we have $|A\times B|= |A| |B|$

\noindent
{\bf Proof of Lemma~\ref{lem:GH}.}
Given $\eps > 0$ we define additional constants $\beta , \eta$ such that
\begin{align}\label{finalhier}
0<\beta \ll \eta \ll \eps , 1/r.
\end{align}
Further assume that $n$ is sufficiently large.

By assumption, either $G = K_{\frac{N}{2}, \frac{N}{2}} \pm \b N^2 $ or  $\overline{G} = K_{\frac{N}{2}, \frac{N}{2}} \pm \b N^2$, where $N:=\binom{n}{2r}$. It suffices to show that if $G = K_{\frac{N}{2}, \frac{N}{2}} \pm 2\b N^2 $ then $H= \mathcal B_{n,4r} \pm \eps n^{4r}$. Indeed, the edge set of $\overline{G}$ contain the edge set of $G(\overline{H})$ and all the pairs of intersecting $2r$-subsets of $V(H)$. Since there are $O(n^{4r-1})$ pairs of intersecting $2r$-subsets of $V(H)$, if $\overline{G} = K_{\frac{N}{2}, \frac{N}{2}}\pm \b N^2$, then $G(\overline{H}) = K_{\frac{N}{2}, \frac{N}{2}}\pm 2\b N^2$, which implies that $\overline{H} = \mathcal B_{n,4r} \pm \eps n^{4r}$, equivalently, $H= \overline{\mathcal B}_{n,4r} \pm \eps n^{4r}$, as desired.
%In both cases $H$ is in the extremal case with parameter $\eps$.

Assume that $G = K_{\frac{N}{2}, \frac{N}{2}} \pm 2 \b N^2$, namely, there is partition $R,B$ of $V(G)=\binom{V(H)}{2r}$ such that $|R| = |B| = N/2$ and $| E(G) \triangle E(K_{R,B}) | \le 2 \b N^2$. Let $K(H)$ denote the complete $2r$-uniform hypergraph whose vertex set is
$V(H)$. Since $R,B$ is a partition of $\binom{V(H)}{2r}$ we may view $R$ and $B$ as the color classes of a $2$-coloring of the edge set of $K(H)$. Let $K[R]$ denote the spanning subhypergraph of $K(H)$ induced by the edges of $R$.
Define $K[B]$ analogously.

Given a  $4r$-set $Q$  of vertices from $V(H)$ we say that $Q$ is \emph{bad} if there exists a partition of $Q= P_1\cup P_2\cup P_3\cup P_4$ where $|P_i|=r$ for all $1\leq i \leq 4$ and
such that exactly one of the four $2r$-sets $P_1 \cup P_2$, $P_2 \cup P_3$, $P_3 \cup P_4$, $P_4 \cup P_1$ receives one of the colors. First assume that this color is $B$. Without loss of generality, assume that $P_1\cup P_2$, $P_2 \cup P_3$, $P_3 \cup P_4\in R$ and $P_4 \cup P_1\in B$.
If $Q \in E(H)$, then $\{P_1\cup P_2, P_3\cup P_4 \}\in E(G)\cap \binom{R}{2}$.
On the other hand, if $Q \not\in E(H)$, then $\{ P_4\cup P_1,  P_2\cup P_3 \} \in E(\overline{G})\cap E(K_{R,B})$.
Therefore, one of $\{ P_1\cup P_2, P_3\cup P_4 \}$ and $\{ P_4\cup P_1,  P_2\cup P_3 \}$ is in $E(G) \triangle E(K_{R,B})$.
The same holds when exactly one of $P_1 \cup P_2$, $P_2 \cup P_3$, $P_3 \cup P_4$, $P_4 \cup P_1$ is colored $R$. Clearly two  distinct bad $4r$-sets lead to two different members of $E(G) \triangle E( K_{R,B})$.
Since $| E(G) \triangle E(K_{R,B}) | \le 2 \b N^2$, the number of bad $4r$-sets is at most $2 \b N^{2}$.

Since $\beta \ll \eta$, we may apply Lemma~\ref{lem:str} to $K(H)$ to obtain that either
$K[R]= \mathcal B_{n,2r} \pm \eta n^{2r}$ or $K[B] = \mathcal B_{n,2r} \pm \eta n^{2r}$.
Since the roles of $K[R]$ and $K[B]$ are interchangeable, we may assume that $K[R]= \mathcal B_{n,2r} \pm \eta n^{2r}$.
Let $X, Y$ denote a partition of $V(H)$ such that $| E(K[R]) \triangle E(\mathcal B_{n,2r}  [X,Y])| \leq \eta n^{2r}$.
We now use the structural information we have about $G$ and $K[R]$ to piece together that of $H$.

\begin{claim}
$H=\mathcal B_{n,4r}  \pm \eps n^{4r}$.
\end{claim}
Recall that given a $2r$-tuple $\underline{x} \in \binom{V(H)}{2r}$ we say that $\underline{x}$ is \emph{even} if $\u{x}$ contains an
even number of elements from $X$ (and so an even number of elements from $Y$). Otherwise, we say that $\u{x}$ is \emph{odd}.
Thus, the edge set of $\mathcal B_{n,2r}  [X,Y]$ is precisely the set of odd $2r$-tuples.

Our ultimate aim is to show that
\begin{align}\label{targeta}
|E(H)\triangle E(\mathcal B_{n,4r}  [X,Y])|\leq \eps n^{4r} .
\end{align}

First we show that $|E(H)\backslash E(\mathcal B_{n,4r}  [X,Y])|\leq \eps n^{4r} /2$. Consider any $4r$-tuple $Q$
from $ E(H)\backslash E(\mathcal B_{n,4r} [X,Y])$.
Since $Q \not \in  E(\mathcal B_{n,4r} [X,Y])$ (thus $|Q\cap X|$ is even), $Q$ can be partitioned into $2r$-tuples $\u{x}$, $\u{y}$ such that both $\u{x}$ and $\u{y}$ are even. (For example, if $|Q\cap X|\ge 2r$, then let $\u{x}$ be a $2r$-subset of $Q\cap X$; otherwise let $\u{x}$ be a $2r$-subset of $Q\cap Y$. Since $|Q\cap X|$ is even, $\u{y}$ is even.) As $Q \in E(H)$ we have that $\{ \u{x} , \u{y} \} \in E(G)$. Thus,
$$ |E(H)\backslash E(\mathcal B_{n,4r} [X,Y])| \leq |\Sigma|,$$
where $\Sigma$ is the set of all disjoint pairs of $2r$-tuples $\u{w}, \u{z} \in \binom{V(H)}{2r}$ such that $\u{w}$ and $\u{z}$ are even
and $\{\u{w} , \u{z} \} \in E(G)$.

Since $K[R] = \mathcal B_{n,2r} [X,Y] \pm \eta n^{2r}$, there are at most $\binom{\eta n^{2r}}{2}\leq \eta ^2 n^{4r}$ pairs
$\{ \u{w},\u{z} \} \in \Sigma$ such that $\u{w},\u{z} \in R$. Similarly, there are at most $\eta n^{2r} |B|\leq \eta n^{4r}$ pairs
$(\u{w},\u{z}) \in \Sigma$ such that $\u{w}\in B$ and $\u{z} \in R$.
Given any pair $(\u{w},\u{z}) \in \Sigma$ such that $\u{w},\u{z} \in B$, by definition of $\Sigma$,  we have $\{\u{w}, \u{z} \} \in E(G)$.
However, $G=K_{R,B} \pm 2 \beta N^2$, so there are most $2 \beta N^2 \leq 2 \beta  n^{4r}$ such pairs in $\Sigma$.
Together, this all implies that $|\Sigma| \leq (\eta ^2 +\eta+2 \beta)n^{4r} \leq \eps n^{4r}/2 $. So indeed,
$|E(H)\backslash E(\mathcal B_{n,4r} [X,Y])|\leq \eps n^{4r}/2 $.

Next we show that $|E(\mathcal B_{n,4r} [X,Y])\backslash E(H)|\leq \eps n^{4r} /2$. Consider any $4r$-tuple
 $Q $ from $
E(\mathcal B_{n,4r} [X,Y])\backslash E(H)$. Since $Q \in  E(\mathcal B_{n,4r} [X,Y])$,
$Q$ can be partitioned into $2r$-tuples $\u{x}$, $\u{y}$
such that $\u{x}$ is even and $\u{y}$ is odd.
(For example, if $|Q\cap X|\ge 2r$, then let $\u{x}$ be a $2r$-subset of $Q\cap X$; otherwise let $\u{x}$ be a $2r$-subset of $Q\cap Y$. Since $|Q\cap X|$ is odd, $\u{y}$ is odd.)
As $Q \not \in E(H)$ we have that $\{ \u{x} , \u{y} \} \in E(\overline{G})$.
Thus,
$$ |E(\mathcal B_{n,4r} [X,Y])\backslash E(H)| \leq |\Gamma|,$$
where $\Gamma$ is the set of all disjoint pairs of $2r$-tuples $\u{w}, \u{z} \in \binom{V(H)}{2r}$ such that $\u{w}$ is even,
$\u{z}$ is odd
and $\{\u{w} , \u{z} \} \in E(\overline{G})$.

Since $K[R] = \mathcal B_{n,2r} [X,Y] \pm \eta n^{2r}$, we have that $K[B] = \overline{\mathcal B}_{n,2r} [X,Y] \pm \eta n^{2r}$.
Thus, there are at most $\eta n^{2r} \binom{n}{2r} \leq \eta n^{4r}$ pairs $\{\u{w},\u{z} \} \in \Gamma$ such that $\u{w}$ is even and $\u{w} \in R$. Similarly, there are at most $\eta n^{4r}$ pairs $\{ \u{w},\u{z} \} \in \Gamma$ such that $\u{z}$ is odd and $\u{z} \in B$.
Given any pair $\{ \u{w},\u{z}\} \in \Gamma$ such that $\u{w}\in R$ is odd and $\u{z} \in B$ is even,
 by definition of $\Gamma$, $\{\u{w}, \u{z} \} \in E(\overline{G})$.
However, $\overline{G}=\overline{K}_{R,B} \pm 2 \beta N^2$, so there are most $2 \beta N^2 \leq 2 \beta  n^{4r}$ such pairs in $\Gamma$.
Together this all implies that $|\Gamma| \leq (2\eta  +\eta+2 \beta)n^{4r} \leq \eps n^{4r}/2$. So indeed,
$|E(\mathcal B_{n,4r} [X,Y])\backslash E(H)|\leq \eps n^{4r}/2 $.  Therefore (\ref{targeta}) is satisfied, as desired.

\medskip
\endproof
\section*{Acknowledgements}
This research was carried out whilst the first author was visiting the Department of Mathematics and Statistics
of Georgia State University. This author would like to thank the department for the hospitality he received.

The authors also thank the referees for their valuable comments that improved the presentation of this paper.

\medskip

{\footnotesize \obeylines \parindent=0pt

\begin{tabular}{lll}

Andrew Treglown                     &\ &  Yi Zhao \\
Computer Science Institute					&\ &  Department of Mathematics and Statistics \\
Faculty of Mathematics and Physics  &\ &  Georgia State University \\
Charles University                  &\ &  Atlanta \\
Malostransk\'e N\'am\v{e}st\'i 25		&\ &  Georgia 30303\\
118 00 Prague												&\ &  USA\\
Czech Republic											&\ &
\end{tabular}
}

{\footnotesize \parindent=0pt

\it{E-mail addresses}:
\tt{treglown@kam.mff.cuni.cz}, \tt{yzhao6@gsu.edu}}

\section*{Appendix}
In this section we prove Theorem~\ref{4thm}. Because of Theorem~\ref{mainthm}, it suffices to prove the following fact.

\begin{fact}\label{4fact} For all $n\geq 12$ divisible by $4$,
$$\delta (n,4,2) \leq \frac{n^2}{4}-\frac{5n}{4} - \frac{\sqrt{n-3}}{2}+\frac{3}{2}.$$
Furthermore, there are infinitely many values of $n$ such that the following holds:
\begin{itemize}
\item $\delta (n,4,2)=\delta_2(\overline{\mathcal B}_{n,4}(t))=\frac{n^2}{4}-\frac{5n}{4} - \frac{\sqrt{n-3}}{2}+\frac{3}{2}$ for some $t$;
\item  $\overline{\mathcal B} _{n,4} (t)$ does not contain a perfect matching.
\end{itemize}
\end{fact}
\proof
Suppose that $n \in \mathbb N$ is divisible by $4$ and let $t$ be an integer such that
$0\leq t<n/2$. Denote the vertex classes of $\mathcal B_{n,4} (t)$ by $A$ and $B$. Therefore $|A|=n/2+t$ and $|B|=n/2-t$.

Given  distinct $v_1, v_2 \in A$,
\begin{align*}
d_{\mathcal B_{n,4} (t)} (v_1v_2) & =(n/2+t-2)(n/2-t)
=\frac{n^2}{4}-n-t^2+2t.
\end{align*}
Given  distinct $w_1, w_2 \in B$,
\begin{align*}
d_{\mathcal B_{n,4} (t)} (w_1w_2)  =(n/2+t)(n/2-t-2)
=\frac{n^2}{4}-n-t^2-2t.
\end{align*}
Given any $v_1 \in A$ and $w_1 \in B$,
\begin{align*}
d_{\mathcal B_{n,4} (t)}(v_1w_1) &= \binom{n/2+t-1}{2} +\binom{n/2-t-1}{2} \\
&= \frac{1}{2} [(n/2+t-1)(n/2+t-2)+(n/2-t-1)(n/2-t-2)] \\
&= \frac{n^2}{4}-\frac{3n}{2}+t^2+2.
\end{align*}
Thus, $d_{\mathcal B_{n,4} (t)}(v_1v_2)\geq d_{\mathcal B_{n,4} (t)} (w_1w_2)$ for all $v_1,v_2 \in A$ and $w_1,w_2 \in B$.
Notice that $n^2/4-n-t^2-2t$ decreases as $t$ increases and that $n^4/4 -3n/2+t^2+2$ increases as $t$ increases (for $t\geq 0$).
For fixed $n$ consider the equation
%\begin{align}\label{eqex2}
\[
\frac{n^2}{4}-n-t_1 ^2 -2t_1 =\frac{n^2}{4}-\frac{3n}{2}+t^2 _1 +2 \ \ \text{ where $t_1 \geq 0$.}
\]
%\end{align} Note that (\ref{eqex2})
It gives that $t^2 _1+t_1+(1-n/4)=0$ and so
$$t_1=\frac{-1+\sqrt{n-3}}{2}.$$
This analysis implies that, for all $0\leq t<n/2$,
\begin{align}\label{tar1}
\delta _2 (\mathcal B_{n,4} (t)) \leq \frac{n^2}{4}-\frac{3n}{2}+t_1 ^2 +2 = \frac{n^2}{4}-\frac{5n}{4}-\frac{\sqrt{n-3}}{2}+\frac{3}{2}.
\end{align}
Further, since $\mathcal B_{n,4} (t)$ is isomorphic to $\mathcal B_{n,4} (-t)$ for all $0\leq t <n/2$, (\ref{tar1}) holds for all $-n/2<t<n/2$.

Now consider $\overline{\mathcal B}_{n,4} (t)$ for any $0\leq t <n/2$ and assume $A$ and $B$ are the vertex classes of
$\overline{\mathcal B}_{n,4} (t)$.
Given  distinct $v_1, v_2 \in A$,
\begin{align*}
d_{\overline{\mathcal B}_{n,4} (t)} (v_1v_2) & =\binom{n/2+t-2}{2}+\binom{n/2-t}{2} \\ &
=\frac{1}{2}\left[ (n/2+t-2)(n/2+t-3)+(n/2-t)(n/2-t-1)\right] \\ &
=\frac{n^2}{4} -\frac{3n}{2}+t^2 -2t +3.
\end{align*}Given  distinct $w_1, w_2 \in B$,
\begin{align*}
d_{\overline{\mathcal B}_{n,4} (t)} (w_1w_2)  =\binom{n/2-t-2}{2}+\binom{n/2+t}{2}
=\frac{n^2}{4} -\frac{3n}{2}+t^2 +2t +3.
\end{align*}
Given any $v_1 \in A$ and $w_1 \in B$,
\begin{align*}
d_{\overline{\mathcal B}_{n,4} (t)} (v_1w_1) =(n/2+t-1)(n/2-t-1)
=\frac{n^2}{4} -n-t^2 +1.
\end{align*}

Notice that $d_{\overline{\mathcal B}_{n,4} (t)} (v_1v_2) \le d_{\overline{\mathcal B}_{n,4} (t)}
 (w_1w_2)$ for all $v_1 ,v_2 \in A$ and $w_1, w_2 \in B$.
Further, when $t \geq 1$, $n^2/4-3n/2+t^2-2t+3$ increases as $t$ increases and that $n^2/4-n-t^2+1$ decreases as $t$ increases.
Thus, for a fixed $n$ the value of $t\ge 1$ which maximizes the minimum $2$-degree of $\overline{\mathcal B}_{n,4} (t)$ satisfies
%\begin{align}\label{stabledeg}
$$n^2/4-3n/2+t^2-2t+3= n^2/4-n-t^2+1,$$
%\end{align} Hence (\ref{stabledeg})
which gives that $t^2 -t +(1-n/4)=0$. Therefore as $t \geq 1$ we have that
$$t=\frac{1+\sqrt{n-3}}{2}.$$
This analysis implies that, for all $1 \leq t<n/2$,
\begin{align}\label{tar2}
\delta _2 (\overline{\mathcal B}_{n,4} (t)) \leq
\frac{n^2}{4}-n-\left(\frac{1+\sqrt{n-3}}{2}\right)^2+1=  \frac{n^2}{4}-\frac{5n}{4}-\frac{\sqrt{n-3}}{2}+\frac{3}{2}.
\end{align}
It is easy to see that $\delta _2 (\overline{\mathcal B}_{n,4} (0)) = \frac{n^2}{4} -\frac{3n}{2}+3 \le \frac{n^2}{4}-\frac{5n}{4}-\frac{\sqrt{n-3}}{2}+\frac{3}{2}$ when $n\ge 12$. Thus \eqref{tar2} holds for all $0 \leq t<n/2$.
Since $\overline{\mathcal B}_{n,4} (t)$ is isomorphic to $\overline{\mathcal B}_{n,4} (-t)$
for all $0\leq t <n/2$, (\ref{tar2}) actually holds for all $-n/2<t<n/2$.
Thus, (\ref{tar1}) and (\ref{tar2}) imply that
$$\delta (n,4,2) \leq \frac{n^2}{4}-\frac{5n}{4}-\frac{\sqrt{n-3}}{2}+\frac{3}{2},$$
as desired.

Notice that there are values of $n$ such that $n$ is divisible by $4$ and where $(1+\sqrt{n-3})/2$ is an odd integer.
Indeed, let $n:=(4m+1)^{2s}+3$ for some $m,s \in \mathbb N$. Then $n = (4m+1)^{2s}+3 \equiv 1+3 \equiv 0 \text{ mod } 4$.
Since $(4m+1)^{s}$ is odd, clearly $(1+\sqrt{n-3})/2= (1+(4m+1)^s)/2$ is an integer. Further
if $(1+(4m+1)^s)/2= 2x$ for some $x \in \mathbb N$ then $(4m+1)^s =4x-1 \equiv 3 \text{ mod } 4$, a contradiction
as $(4m+1)^s \equiv 1 \text{ mod } 4$. Hence $(1+\sqrt{n-3})/2$ is  odd.

For values of $n$ where $n$ is divisible by $4$ and where $t:=(1+\sqrt{n-3})/2$ is an odd integer, we have that
$$\delta _2(\overline{\mathcal B}_{n,4} (t)) = \frac{n^2}{4}-n-t^2+1= \frac{n^2}{4}-\frac{5n}{4} - \frac{\sqrt{n-3}}{2}+\frac{3}{2}.$$
Note though that $|A|=n/2+t $ is odd, therefore,  $\overline{\mathcal B}_{n,4} (t) \in \mathcal H_{\text{ext}} (n,4)$ and so it does
not contain a perfect matching.
Thus, the second part of Fact~\ref{4fact} is proven.\endproof

\end{document}